\documentclass[11pt]{amsart}
\usepackage{verbatim}

\newtheorem{theorem}{Theorem}[section]
\newtheorem{proposition}[theorem]{Proposition}

\newtheorem{lemma}[theorem]{Lemma}
\newtheorem{claim}[theorem]{Claim}

\newtheorem{definition}[theorem]{Definition}

\newtheorem{conjecture}[theorem]{Conjecture}

\theoremstyle{plain}
\numberwithin{equation}{theorem}

\theoremstyle{remark}
\newtheorem{remark}[theorem]{Remark}

\newcommand{\oQ}{{\overline{\mathbb Q}}}

\newcommand{\Z}{{\mathbb Z}}

\newcommand{\cV}{{\mathcal V}}
\newcommand{\cW}{{\mathcal W}}
\newcommand{\cK}{{\mathcal K}}

\newcommand{\fO}{\mathfrak o}
\newcommand{\cC}{\mathfrak C}
\newcommand{\fo}{\mathfrak o}

\newcommand{\Kbar}{\overline{K}}

\DeclareMathOperator{\N}{\mathbb{N}}

\DeclareMathOperator{\trdeg}{trdeg}

\DeclareMathOperator{\car}{char}

\DeclareMathOperator{\End}{End}

\DeclareMathOperator{\Supp}{Supp}

\DeclareMathOperator{\bN}{\mathbb{N}}

\DeclareMathOperator{\PGL}{PGL}

\newcommand{\bP}{{\mathbb P}}
\newcommand{\bZ}{{\mathbb Z}}
\newcommand{\bG}{{\mathbb G}}

\newcommand{\bC}{{\mathbb C}}

\newcommand{\bA}{{\mathbb A}}
\newcommand{\bQ}{{\mathbb Q}}

\newcommand{\lra}{\longrightarrow}
\newcommand{\cO}{\mathcal{O}}

\newcommand{\cN}{\mathcal{N}}
\newcommand{\cU}{\mathcal{U}}

\newcommand{\cS}{\mathcal{S}}

\newcommand{\bK}{\overline{K}}

\newcommand{\Cp}{\bC_p}
\newcommand{\PCp}{\bP^1(\bC_p)}

\newcommand{\Dbar}{\overline{D}}

\DeclareMathOperator{\lcm}{lcm}

\title[Dynamical Mordell-Lang]
      {The Dynamical Mordell-Lang Conjecture}

\author[Benedetto, Ghioca, Kurlberg, Tucker]
{Robert L.  Benedetto, Dragos Ghioca, P\"{a}r Kurlberg,
  and Thomas J. Tucker}

\date{February 4, 2009}



\address{Department of Mathematics \& Computer Science \\
        Amherst College \\
        Amherst, MA 01002 \\
        USA}
\email{rlb@cs.amherst.edu}

\address{Department of Mathematics \& Computer Science\\
University of Lethbridge\\
Lethbridge, AB T1K 3M4\\
Canada}
\email{dragos.ghioca@uleth.ca}

\address{
Department of Mathematics\\
KTH\\
SE-100 44 Stockholm\\
Sweden}
\email{kurlberg@math.kth.se}

\address{Department of Mathematics\\
University of Rochester\\
Rochester, NY 14627\\ 
USA}
\email{ttucker@math.rochester.edu}

\begin{document}

\begin{abstract}
  We prove a special case of a dynamical analogue of the classical
  Mordell-Lang conjecture. In particular, let $\varphi$ be a rational
  function with no superattracting periodic points other than
  exceptional points.  If the coefficients of $\varphi$ are algebraic, we
  show that the orbit of a point outside the union of proper
  preperiodic subvarieties of $(\bP^1)^g$ has only finite intersection
  with any curve contained in $(\bP^1)^g$. We also show that our result holds for indecomposable polynomials $\varphi$ with coefficients in $\bC$. 
 Our proof uses results
  from $p$-adic dynamics together with an integrality argument. The extension to polynomials defined over $\bC$ uses the method of specializations coupled with some new results of Medvedev and Scanlon for describing the periodic plane curves under the action of $(\varphi,\varphi)$ on $\bA^2$.
\end{abstract}

\maketitle

\section{Introduction}
Let $X$ be a variety over the complex numbers $\bC$, let $\Phi: X \lra
X$ be a morphism, and let $V$ be a subvariety of
$X$.  For any integer $m\geq 0$, denote by $\Phi^m$ the
$m^{\text{th}}$ iterate $\Phi\circ\cdots\circ\Phi$.
If $\alpha \in X(\bC)$ has the property that there is some
integer $\ell\geq 0$ such that $\Phi^\ell(\alpha) \in W(\bC)$,
where $W$ is a periodic subvariety of $V$, then there are infinitely
many integers
$n\geq 0$ such that $\Phi^n(\alpha) \in V$.  More precisely, if
$k\geq 1$ is the period of $W$ (the smallest positive integer $m$ for which
$\Phi^m(W) = W$), then $\Phi^{nk + \ell}(\alpha) \in W(\bC)
\subseteq V(\bC)$ for all integers $n\geq 0$.  It is natural
then to pose the following question: given
$\alpha\in X(\bC)$, if there are infinitely many
integers $m\geq 0$ such that $\Phi^m(\alpha) \in V(\bC)$,
are there necessarily integers $k\geq 1$ and $\ell\geq 0$
such that $\Phi^{nk + \ell}(\alpha) \in V(\bC)$ for all
integers $n\geq 0$?

This question has a positive answer in many special cases.  When $X$
is a semiabelian variety and $\Phi$ is a multiplication-by-$m$ map,
this follows from Faltings' \cite{Faltings} and Vojta's proof
\cite{V1} of the Mordell-Lang conjecture in characteristic $0$.  More
generally, the question has a positive answer when $\Phi$ is any
endomorphism of a semiabelian variety (see \cite{GT-newlog}).  Denis
\cite{Denis-dynamical} treated the general question under the
additional hypothesis that the integers $n$ for which $\Phi^n(\alpha)
\in V(\bC)$ are sufficiently dense in the set of all positive
integers; he also obtained results for automorphisms of projective
space without using this additional hypothesis.  Bell \cite{Bell}
later solved the problem completely in the case of automorphisms of
affine space.  In \cite{GT-newlog}, a general framework for attacking the
problem was developed and the following conjecture was made.

\begin{conjecture}
\label{dynamical M-L}
Let $f_1,\dots,f_g\in\bC[t]$ be polynomials, let $\Phi$ be their
action coordinatewise on $\bA^g$, let $\cO_\Phi((x_1,\dots,x_g))$
denote the $\Phi$-orbit of
the point $(x_1,\dots,x_g)\in \bA^g(\bC)$, and let $V$ be a subvariety
of $\bA^g$.  Then $V$ intersects $\cO_\Phi((x_1,\dots,x_g))$ in at
most a finite union of orbits of the form
$\cO_{\Phi^k}(\Phi^{\ell}(x_1,\dots,x_g))$, for some nonnegative
integers $k$ and $\ell$.
\end{conjecture}
See Section~\ref{notation} for the definition of the orbit
$\cO_{\Phi}(\alpha)$.
Note that the orbits for which $k=0$ are singletons, so that the
conjecture allows not only infinite forward orbits but also finitely
many extra points.

Note also that if Conjecture~\ref{dynamical M-L} holds
for a given map $\Phi$, variety $V$, and
non-preperiodic point $\alpha=(x_1,\ldots,x_g)$,
and if $V$ intersects the $\Phi$-orbit of $\alpha$
in infinitely many points, then $V$ must contain a positive-dimensional
subvariety $V_0$ that is periodic under $\Phi$.  Indeed,
the conjecture says that there are integers
$k\geq 1$ and $\ell\geq 0$ such that
$\Phi^{nk+\ell}(\alpha)$ lies on $V$ for all $n\geq 0$.
Since $\alpha$ is not preperiodic, the set
$S=\{\Phi^{nk+\ell}(\alpha)\}_{n\geq 0}$ is infinite,
and therefore its Zariski closure $V'_0$ contains positive-dimensional
components.  Thus, if we let $V_0$ be the union of
the positive-dimensional irreducible subvarieties of $V_0'$,
then $V_0$ is positive-dimensional and fixed by $\Phi^k$, as claimed.


Conjecture~\ref{dynamical M-L} fits into Zhang's far-reaching system
of dynamical conjectures \cite{ZhangLec}.  Zhang's conjectures include
dynamical analogues of the Manin-Mumford and Bogomolov conjectures for
abelian varieties (now theorems of Raynaud \cite{Raynaud1, Raynaud2},
Ullmo \cite{Ullmo}, and Zhang \cite{Zhang}), as well as a conjecture
about the Zariski density of orbits of points under fairly general
maps from a projective variety to itself. This latter conjecture of
Zhang takes the following form in the case of polynomial actions on
$\bA^g$.

\begin{conjecture}
\label{Zhang-conjecture}
Let $f_1,\dots,f_g\in \overline{\mathbb{Q}}[t]$ be polynomials of the same degree
$d\geq 2$, and
let $\Phi$ be their action coordinatewise on $\bA^g$. Then there is
a point $(x_1,\dots,x_g)\in\bA^g(\overline{\mathbb{Q}})$ such that
$\cO_{\Phi}((x_1,\dots,x_g))$ is Zariski dense in $\bA^g$.
\end{conjecture}

Conjectures \ref{Zhang-conjecture} and \ref{dynamical M-L} may be
thought of as complementary.  Conjecture~\ref{Zhang-conjecture} posits
that there is a point in $\bA^g$ outside the union of the
preperiodic proper subvarieties of $\bA^g$ under the action of $\Phi$,
while Conjecture~\ref{dynamical M-L} asserts if a point $\alpha$ lies
outside this union of preperiodic subvarieties, then the orbit of
$\alpha$ under $\Phi$ intersects any subvariety $V$ of $\bA^g$ in at
most finitely many points. We view our Conjecture~\ref{dynamical M-L} as an analogue of the classical Mordell-Lang conjecture for arithmetic dynamics where groups of rank one are replaced by single orbits. We also note that a stronger form of Conjecture~\ref{Zhang-conjecture} was proved in \cite[Theorem 5.11]{Medvedev/Scanlon}.


In this paper, we prove Conjecture~\ref{dynamical M-L} over number
fields for curves embedded in $\bA^g$ under the diagonal action of any
polynomial which has no periodic superattracting points. (Roughly
speaking, a superattracting periodic point is a periodic point at
which the derivative vanishes; for a formal
definition of superattracting points, see Section~\ref{notation}.)
In fact, we prove the following more general statement.

\begin{theorem}\label{curves}
  Let $C\subset \left(\bP^1\right)^g$ be a curve defined over $\oQ$,
  and let $\Phi:=(\varphi,\dots,\varphi)$ act on $\left(\bP^1\right)^g$
  coordinatewise, where $\varphi\in \oQ(t)$ is a rational function with
  no periodic superattracting points other than exceptional
  points. Let $\cO$ be the $\Phi$-orbit of a point
  $(x_1,\dots,x_g)\in\left(\bP^1\right)^g(\oQ)$. Then $C(\oQ)\cap\cO$
  is a union of at most finitely many orbits of the form
  $\{\Phi^{nk+\ell}(x_1,\dots,x_g)\}_{n\ge 0}$ for nonnegative
  integers $k$ and $\ell$.
\end{theorem}
See Section~\ref{notation} for a definition of exceptional
points.

Using recent results of 
Medvedev and Scanlon \cite{Medvedev/Scanlon}
from model theory and polynomial decomposition, we will extend Theorem~\ref{curves}
to the complex numbers, at least under the action
of indecomposable polynomials.  (See Definition~\ref{def:indecomp}.) Our method from Section~\ref{extensions to C} also extends to the case of any polynomials with complex coefficients, as long as they do not have periodic superattracting points other than exceptional points (see Remark~\ref{possible extensions}).

\begin{theorem}
\label{polynomials, arbitrary curves over C}
Let $\varphi\in\bC[t]$ be an indecomposable polynomial
with no periodic superattracting points other than exceptional points,
and let $\Phi$ be its diagonal action on
$\bA^g$ (for some $g\ge 1$). Let $\cO$ be the $\Phi$-orbit of a point
$P$ in $\bA^g(\bC)$, and let $C$ be a curve defined over $\bC$. Then
$C(\bC)\cap\cO$ is at most a finite union of orbits of the form
$\{\Phi^{nk+\ell}(P)\}_{n\ge 0}$, for some nonnegative integers $k$
and $\ell$.
\end{theorem}

When the function $\varphi$ is a quadratic polynomial, we can prove a
similar result for subvarieties of any dimension.
\begin{theorem}\label{quadratic}
  Let $V\subset \left(\bP^1\right)^g$ be a subvariety defined over
  $\oQ$, and let $\Phi:=(f,\dots,f)$ act on $\left(\bP^1\right)^g$
  coordinatewise, where $f \in \oQ[t]$ is a quadratic polynomial with no
  periodic superattracting points in $\oQ$. Let $\cO$ be the
  $\Phi$-orbit of a point
  $(x_1,\dots,x_g)\in\left(\bP^1\right)^g(\oQ)$. Then $V(\oQ)\cap\cO$
  is a union of at most finitely many orbits of the form
  $\{\Phi^{nk+\ell}(x_1,\dots,x_g)\}_{n\ge 0}$ for nonnegative
  integers $k$ and $\ell$.
\end{theorem}

For quadratic polynomials over the rational numbers, we can remove the
hypothesis on superattracting points and obtain a 
stronger result.

\begin{theorem}\label{rationals}
  Let $V\subset \left(\bP^1\right)^g$ be a subvariety defined over
  $\bQ$, and let $\Phi:=(f,\dots,f)$ act on $\left(\bP^1\right)^g$
  coordinatewise, where $f \in \bQ[t]$ is a quadratic polynomial.
  Let $\cO$ be the $\Phi$-orbit of a point
  $(x_1,\dots,x_g)\in\left(\bP^1\right)^g(\bQ)$. Then $V(\bQ)\cap\cO$
  is a union of at most finitely many orbits of the form
  $\{\Phi^{nk+\ell}(x_1,\dots,x_g)\}_{n\ge 0}$ for nonnegative
  integers $k$ and $\ell$.
\end{theorem}

Using results of Jones~\cite{Jones}, we can prove the
corresponding result for maps of the form
$\Phi =(f_1,\dots,f_g)$, without the restriction that $f_i=f_j$,
if each $f_j$ is of the form $f_j(t)=t^2 + c_j$ with $c_j\in\Z$.

\begin{theorem}\label{rationals-2}
  Let $V\subset \left(\bP^1\right)^g$ be a subvariety defined over
  $\bQ$, and let $\Phi:=(f_1,\dots,f_g)$ act on $\bA^g$
  coordinatewise, where $f_i(t)=t^2 + c_i$ with $c_i\in\Z$ for each $i$.
  Let $\cO$ be the $\Phi$-orbit of a point
  $(x_1,\dots,x_g)\in (\bZ)^g$. Then $V(\bQ)\cap\cO$
  is a union of at most finitely many orbits of the form
  $\{\Phi^{nk+\ell}(x_1,\dots,x_g)\}_{n\ge 0}$ for nonnegative
  integers $k$ and $\ell$.
\end{theorem}


The strategy used for proving the theorems above involves an
interplay between arithmetic geometry and $p$-adic dynamics, and
it is
based in part on a non-linear analog of the technique used by Skolem
\cite{Skolem} (and later extended by Mahler \cite{Mahler-2} and Lech
\cite{Lech}) to treat linear recurrence sequences.
However, unlike in the linear recurrence case (where all but finitely many
$p$-adic absolute values will work), finding a suitable prime $p$ is
far more difficult and involves using intersection theory on
$\bP^1\times \bP^1$ together with an application of the classical
Siegel's theorem (see Section~\ref{intersection theory}.)
A further complication is the problem of finding fixed points around
which the dynamics can be linearized; instead,
we invoke the work of
Rivera-Letelier \cite{Riv} from $p$-adic dynamics,
as described in Section~\ref{dynamics}.
More precisely, we find
arithmetic progression $\cS$ of integers such that there are infinitely
many $m \in \cS$ with $\Phi^m(\alpha)$ lying on $V$, and then we
construct a $p$-adic analytic map $\theta$ sending
$\cS$ into $\bA^g(\bC_p)$ such that
$\theta(m) = \Phi^m(x_1, \dots,x_g)$
for each integer $m$ in the sequence.  Then, for any polynomial
$F$ that vanishes on $V$, we have $F(\theta(k))=0$ for infinitely many $k$.
Since the zeros of a nonzero $p$-adic analytic function are isolated,
$F\circ \theta$ must vanish at {\it all} $k$ in the sequence.
Rivera-Letelier's results, which are used in the construction of $\theta$,
apply whenever there is a positive integer $\ell$ such that
$\varphi^\ell (x_i)$ is in a $p$-adic quasiperiodicity disk
for each $i$. (A quasiperiodicity disk is a periodic residue class on
which the derivative has absolute value equal to one;
see Section~\ref{dynamics} for a formal definition.)
However, one cannot expect every place to admit such an integer
$\ell$, but in our case the above mentioned diophantine techniques
can be used to show that at least one such place exists.

We note that the Skolem-Mahler-Lech technique has played a role in
other work done on this subject.  Bell's \cite{Bell} and Denis's
\cite{Denis-dynamical} work on automorphisms may be viewed as
algebro-geometric realizations of the Skolem-Mahler-Lech theorem.
 Evertse, Schlickewei, and Schmidt \cite{ESS} have
given a strong quantitative version of the Skolem-Mahler-Lech theorem.
It may be possible to use their result to give more precise versions
of the theorems of this paper.

%
We remark that Conjecture~\ref{dynamical M-L} has been proved
(cf. \cite{Mike}) by conceptually quite different methods in the
special case that $g=2$ and $V$ is a line in $\bA^2$.  However, the
methods used there (involving Ritt's classification for
functional decomposition of complex polynomials) do not appear to
work for more general subvarieties of affine space.


We exclude the case that the rational function $\varphi$ has
superattracting points because we have been unable, thus far, to
extend the method of Skolem-Mahler-Lech to this situation.
Although there is a
logarithm associated to a $\varphi$ (see \cite{GT-newlog}) in a
neighborhood of a superattracting point, it does not have the
required properties that
Rivera-Letelier's logarithms for quasiperiodicity disks
\cite{Riv} provide.
It should not be surprising that superattracting points
pose difficulties; while they are relatively simple
dynamically, they cause ramification issues in a
diophantine context.  In the cases of 
endomorphisms of semiabelian varieties
(see \cite{V1, Faltings, GT-newlog})
and of automorphisms of affine space (see \cite{Denis-dynamical, Bell}), the
underlying maps are {\'e}tale and hence have no ramification.  
In fact, it is possible to prove a very general dynamical
Mordell-Lang theorem for unramified maps (see \cite{BGT}) without using
the techniques from diophantine approximation that appear in Sections~4 and 6.  Thus, at this point, it seems the main obstacle to proving
Conjecture~1.1 is overcoming the difficulties that ramification
presents.
When
$\varphi$ has no superattracting points, the ramification indices of $\varphi^n$ remain bounded for all $n$; this fact plays an
important role in Section~\ref{intersection theory}.
However, when $\varphi$
has a superattracting point, these indices may become arbitrarily
large.  Hence,
the ramification of the iterates of
$\varphi$ is more complicated when $\varphi$ has a superattracting
point.


In general, we believe
that there should be a broader \emph{Mordell-Lang principle} which
holds for any sufficiently \emph{rigid} space $X$ (i.e. the space does
not have a large set of endomorphisms). This principle would say that
any definable subset of $X$ (in the sense of model
theory; for algebraic geometry, the definable sets are algebraic
varieties) intersects the orbit of a point $P\in X$ under an
endomorphism $\Phi$ of $X$ in at most finitely many orbits of the form
$\{\Phi^{nk+\ell}(P)\}_{n\ge 0}$, for some nonnegative integers $k$
and $\ell$. If $X$ is a semiabelian variety, the above principle can
be found at the heart of the classical Mordell-Lang conjecture (see
\cite{GT-newlog}). If $X$ is $\bA^g$ under the action of polynomial maps
$f_i$ on each coordinate, then we recover our
Conjecture~\ref{dynamical M-L}.
Note that in either case, $X$ has few
endomorphisms.  If $X$ is semiabelian, then $\End(X)$ is a finitely
generated integral extension of $\mathbb{Z}$.  Similarly, if $X$ is
$\bA^g$ under a coordinatewise polynomial action, then $(H_1,\dots,H_g)$ is an
endomorphism if and only if $H_i\circ f_i = f_i\circ H_i$ for each
$i$, which typically implies that $H_i$ and $f_i$ have a common
iterate.  (See the extensive work on this subject by Fatou
\cite{Fatou-1, Fatou-2}, Julia \cite{Julia}, Eremenko
\cite{Eremenko}, among many others).

The outline of our paper is as follows.
In Section~\ref{notation} we
introduce our notation.  Sections~\ref{dynamics} and
\ref{intersection theory}
provide necessary lemmas from $p$-adic
dynamics and from intersection theory on arithmetic surfaces.
In Section~\ref{sect:dml}, we prove
Theorem~\ref{curves}.  In
Section~\ref{sect: quadratic}, we
use the results of Section~\ref{sect:dml}
to prove Theorems~\ref{quadratic}, \ref{rationals}.
and~\ref{rationals-2}.
Finally, in Section~\ref{extensions to C}, we
describe the results of \cite{Medvedev/Scanlon}
and use them to deduce
Theorem~\ref{polynomials, arbitrary curves over C}.


\smallskip
\noindent{\bf Acknowledgements.}
 The authors would like to thank  A.~Medvedev, T.~Scanlon, L.~Szpiro, and
  M.~Zieve for helpful conversations.  
Research of R.~B. was partially supported by NSF Grant
  DMS-0600878, that of D.~G. by an NSERC grant, that of P.~K.  by grants from the G\"oran Gustafsson
  Foundation, the Royal Swedish Academy of Sciences, and the Swedish
  Research Council, and that of T.~T. by NSA Grant 06G-067.  

\section{Notation}
\label{notation}
  We write 
  $\N$ for the set of nonnegative integers.
If $K$ is a field, we write $\overline{K}$ for an
  algebraic closure of $K$.
Given a prime number $p$, the field $\Cp$ will denote the
completion of an algebraic closure $\overline{\bQ}_p$
of $\bQ_p$, the field of $p$-adic
rationals.
We denote by $\mid\cdot\mid:=\mid\cdot\mid_p$ the usual absolute value
on $\Cp$.
  Given $a\in\Cp$ and $r>0$,
we write $D(a,r)$ and $\Dbar(a,r)$
\label{diskpage}
for the open disk and closed disk (respectively) of radius
$r$ centered at $a$.


If $K$ is a number field, we let $\fo_K$ be its ring of algebraic
integers, and we fix an isomorphism $\pi$ between $\bP^1_K$ and the
generic fibre of $\bP^1_{\fo_K}$. For each
nonarchimedean place $v$ of $K$, we let
$k_v$ be the residue field of $K$ at $v$, and for each $x\in\bP^1(K)$,
we let $x_v:=r_v(x)$ be the intersection of the Zariski closure of
$\pi(x)$ with the fibre above $v$ of $\bP^1_{\fo_K}$.
(Intuitively, $x_v$ is $x$ modulo $v$.)
This map $r_v:\bP^1(K)\lra\bP^1(k_v)$ is the \emph{reduction map} at $v$.


If $\varphi:\bP^1\to\bP^1$
is a morphism defined over the field $K$,
then (fixing a choice of homogeneous coordinates)
there are relatively prime homogeneous polynomials $F,G\in K[X,Y]$
of the same degree $d=\deg\varphi$ such that
$\varphi([X,Y])=[F(X,Y):G(X,Y)]$.  
(In affine coordinates, $\varphi(t)=F(t,1)/G(t,1)\in K(t)$
is a rational function in one variable.)
Note that by our choice of coordinates,
$F$ and $G$ are uniquely defined up to a nonzero constant
multiple.
We will need the notion of good reduction of $\varphi$,
first introduced by Morton and Silverman in~\cite{MorSil1}.

\begin{definition}
\label{good reduction}
Let $K$ be a field, let $v$ be a nonarchimedean valuation on
$K$, let $\fO_v$ be the ring of $v$-adic integers of $K$,
and let $k_v$ be the residue field at $v$.
Let
$\varphi:\bP^1\lra\bP^1$ be a morphism over $K$,
given by $\varphi([X,Y])=[F(X,Y):G(X,Y)]$, where
$F,G\in \fO_v[X,Y]$ are relatively prime homogeneous
polynomials of the same degree
such that at least one coefficient
of $F$ or $G$ is a unit in $\fO_v$.
Let $\varphi_v :=[F_v,G_v]$, where
$F_v,G_v\in k_v[X,Y]$ are the reductions
of $F$ and $G$ modulo $v$.
We say that $\varphi$ has
{\em good reduction} at $v$ if
$\varphi_v:\bP^1(k_v)\lra\bP^1(k_v)$ is a morphism of the same
degree as $\varphi$.
\end{definition}
If $\varphi\in K[t]$ is a polynomial, we can give
the following elementary criterion for good reduction:
$\varphi$ has good reduction at $v$ if and only if all coefficients of $\varphi$ are
$v$-adic integers, and its leading coefficient is a $v$-adic unit.

\begin{definition}
\label{conjugated rational maps}
Two rational functions $\varphi$ and $\psi$ are {\em conjugate} if
there is a linear fractional transformation $\mu$ such that $\varphi =
\mu^{-1}\circ \psi \circ\mu$.
\end{definition}
In the above definition, if $\varphi$ and $\psi$ are polynomials,
then we may assume that $\mu$ is a polynomial of degree one.

\begin{definition}
If $K$ is a field, and $\varphi\in K(t)$ is a
rational function,
then $z\in\bP^1(\Kbar)$ is a {\em periodic point} for
$\varphi$ if there exists an integer $n\geq 1$ such that
$\varphi^n(z) = z$.
The smallest such integer $n$ is the {\em period}
of $z$, and $\lambda=(\varphi^n)'(z)$ is
the {\em multiplier} of $z$.
If $\lambda=0$, then $z$ is called {\em superattracting}.
If $|\cdot |_v$ is an absolute value on $K$,
and if $|\lambda|_v < 1$, then $z$ is called {\em attracting}.
\end{definition}

If $z$ is a periodic point of $\varphi=\mu^{-1}\circ\psi\circ\mu$,
then $\mu(z)$ is
a periodic point of $\psi$ with the same multiplier.
In particular, we can define the multiplier of a periodic point
at $z=\infty$ by changing coordinates.

Whether or not $z$ is periodic, 
we say $z$ is a {\em ramification point} or {\em critical point}
of $\varphi$ if $\varphi'(z)=0$.  If
$\varphi=\mu^{-1}\circ\psi\circ\mu$, then $z$ is a critical point
of $\varphi$ if and only if $\mu(z)$ is a critical point of $\psi$;
in particular, coordinate change can again be used to determine
whether $z=\infty$ is a critical point.
Note that a periodic point $z$ is superattracting if
and only if at least one of $z,\varphi(z),\varphi^2(z),\ldots,\varphi^{n-1}(z)$
is critical, where $n$ is the period of $z$.

Let $\varphi: V \lra V$ be a map from a variety to itself,
and let $z\in V(\Kbar)$.
The ({\em forward}) {\em orbit}
$\cO_{\varphi}(z)$ of $z$ under
$\varphi$ is the set $\{ \varphi^k(z)\text{ : }k\in\N\}$.
We say $z$ is \emph{preperiodic} if $\cO_{\varphi}(z)$ is finite.
If $\mu$ is an automorphism of $V$,
and if $\varphi=\mu^{-1}\circ\psi\circ\mu$,
note that
$\cO_{\varphi}(z) = \mu^{-1}(\cO_\psi(\mu(z)))$.

We say $z$ is \emph{exceptional}
(or \emph{totally invariant})
if there are only finitely many points $w$ such that
$z\in\cO_{\varphi}(w)$ (i.e. the backward orbit of $z$ contains only
finitely many points).
It is a classical result in dynamics (e.g., see \cite{beardon-book},
Theorem~4.1.2) that a 
morphism $\varphi:\bP^1\rightarrow\bP^1$
of degree larger than one has at most two exceptional points.  
Moreover, it has exactly two if and only if $\varphi$ is conjugate
to the map $t\mapsto t^n$, for some integer
$n\in\bZ$; and it has exactly one if and only if $\varphi$ is conjugate to a
polynomial but not to any map $t\mapsto t^n$.
In particular, $\varphi$ has at least one exceptional point
if and only if $\varphi^2$ is conjugate to a polynomial.
%
%
%

\section{Quasiperiodicity disks in $p$-adic dynamics}
\label{dynamics}

As in \cite{GT-newlog}, we will need a result on non-preperiodic points over
local fields.  
By an {\em open disk in $\PCp$}, we will mean either an open disk in
$\Cp$ or the complement (in $\PCp$) of a closed disk in $\Cp$.
Equivalently, an open disk in $\PCp$ is the image of an open disk
$D(0,r)\subseteq\Cp$ under a linear fractional transformation
$\gamma\in\PGL(2,\Cp)$.  Closed disks are defined similarly.

The following definition is borrowed from
\cite[Section~3.2]{Riv}, although we have used a simpler version
that suffices for our purposes.

\begin{definition}
\label{def:siegel}
Let $p$ be a prime, let $r>0$, let $\gamma\in\PGL(2,\Cp)$,
and let $U=\gamma(D(0,r))$.  Let $f:U\rightarrow U$ be
a function such that
$$\gamma^{-1}\circ f \circ \gamma(t) = \sum_{i\geq 0} c_i t^i
\in\Cp[[t]],$$
with $|c_0|<r$, $|c_1|=1$, and $|c_i|r^i \leq r$ for all $i\geq 1$.
Then we say $U$ is a {\em quasiperiodicity disk} for $f$.
\end{definition}

The conditions on $f$ in
Definition~\ref{def:siegel} mean precisely that
$f$ is rigid analytic and maps $U$ bijectively onto $U$.
In particular, the preperiodic
points of $f$ in $U$ are in fact periodic.
By \cite[Corollaire~3.12]{Riv},
our definition implies that $U$ is indeed
a quasiperiodicity domain of $f$
in the sense of \cite[D\'{e}finition~3.7]{Riv}.  

The main result of this section is the following.

\begin{theorem}
\label{indifferent}
  Let $p$ be a prime and $g\geq 1$.  For each $i=1,\ldots, g$,
  let $U_i$ be an open disk in $\PCp$, and let $f_i:U_i\to U_i$
  be a map for which $U_i$ is a quasiperiodicity disk.
  Let $\Phi$ denote
  the action of $f_1\times \cdots\times f_g$ on
  $U_1\times\cdots \times U_g$, let
  $\alpha=(x_1,\ldots,x_g)\in U_1\times\cdots\times U_g$ be a point,
  and let $\cO$ be the $\Phi$-orbit of $\alpha$.
  Let $V$ be a subvariety of $(\bP^1)^g$ defined over $\Cp$.
  Then $V(\Cp)\cap\cO$ is a union of at most finitely many
  orbits of the form $\{\Phi^{nk+\ell}(\alpha)\}_{n\geq 0}$
  for nonnegative integers $k$ and $\ell$.
\end{theorem}


The proof of Theorem~\ref{indifferent} relies on
the following lemma from $p$-adic dynamics,
which in turn follows from the theory
of quasiperiodicity domains in 
\cite[Section~3.2]{Riv}.

\begin{lemma}
\label{lem:siegel}
Let $U\subseteq\Cp$ be an open disk, let
$f:U\rightarrow U$
be a map for which $U$ is a quasiperiodicity disk,
and let $x\in U$ be a non-periodic point.
Then there exist an integer $k\geq 1$,
radii $r>0$ and $s\geq |k|_p$,
and, for every integer $\ell\geq 0$,
a bijective rigid analytic function
$h_{\ell}:\Dbar(0,s)\rightarrow \Dbar(f^{\ell}(x),r)$,
with the following properties:
\begin{enumerate}
\item $h_{\ell}(0)=f^{\ell}(x)$, and
\item for all $z\in \Dbar(f^{\ell}(x),r)$ and $n\geq 0$, we have
$$f^{nk}(z) = h_{\ell}(nk+ h_{\ell}^{-1}(z)).$$
\end{enumerate}
\end{lemma}

\begin{proof}
Write $U=D(a,R)$.
By \cite[Proposition~3.16(2)]{Riv}, there is an integer $k\geq 1$
and a neighborhood $U_x\subseteq U$ of $x$ on which $f^k$
is (analytically and bijectively) conjugate to $t\mapsto t+k$.  That is,
there are radii $r,s>0$ (with $r< R$ and $s\geq |k|_p$)
and a bijective analytic
function $h_{0}:\Dbar(0,s)\rightarrow \Dbar(x,r)$
such that $f^{nk}(z) = h_0(nk+ h_0^{-1}(z))$
for all $z\in \Dbar(x,r)$ and $n\geq 0$.

For each nonnegative integer $\ell$, note that $f^{\ell}$ is a bijective
analytic function from $\Dbar(x,r)$ onto $\Dbar(f^{\ell}(x),r)$.
Thus, if we let $h_{\ell} := f^{\ell}\circ h_0$,
then $h_{\ell}$ is a bijective analytic function
from $\Dbar(0,s)$ onto $\Dbar(f^{\ell}(x),r)$.  Moreover,
for all $z\in \Dbar(f^{\ell}(x),r)$, if we let
$\zeta = f^{-\ell}(z)\in \Dbar(x,r)$, then for every $n\ge 0$,
$$f^{nk}(z) = f^{\ell}(f^{nk}(\zeta))
= f^{\ell}(h_0(nk + h_0^{-1}(\zeta)))
= h_{\ell}(nk + h_{\ell}^{-1}(z)).$$
Finally, replacing 
$h_{\ell}(z)$ by $h_{\ell}(z + h_{\ell}^{-1}(f^{\ell}(x)))$,
we can also ensure that $h_{\ell}(0)=f^{\ell}(x)$.
\end{proof}

We are now ready to prove Theorem~\ref{indifferent}.
\begin{proof}[Proof of Theorem~\ref{indifferent}]
By applying linear fractional transformations $\gamma_i$ to each
$U_i$, we may assume without loss of generality that each $U_i$ is an open
disk in $\Cp$. 

For each $i=1,\ldots, g$, consider the $f_i$-orbit of $x_i$.
If $x_i$ is periodic, let $k_i\geq 1$ denote its period,
and for every $\ell\geq 0$, define the
power series $h_{i,\ell}$ to be the constant $f_i^{\ell}(x_i)$.
Otherwise, choose
$k_i\geq 1$ and radii $r_i,s_i>0$
according to Lemma~\ref{lem:siegel}, along with
the associated conjugating maps $h_{i,\ell}$
for each $\ell\ge 0$.

Let $k=\lcm(k_1,\ldots,k_g)\geq 1$.
For each $\ell\in\{0,\ldots,k-1\}$ such that
$V(\Cp)\cap \cO_{\Phi^k}(\Phi^{\ell}(\alpha))$ is finite,
we can cover $V(\Cp)\cap \cO_{\Phi^k}(\Phi^{\ell}(\alpha))$
by finitely many singleton orbits. 

It remains to consider those $\ell\in\{0,\ldots,k-1\}$ for which
there is an infinite set $\cN$ of nonnegative integers $n$ such that
$\Phi^{nk+\ell}(\alpha)\in V(\Cp)$. We will show that in fact,
$\Phi^{nk+\ell}(\alpha)\in V(\Cp)$ for \emph{all} $n\in\N$.

For any $|z|\leq 1,$ note that $kz\in\Dbar(0,s_i)$ for
all $i=1,\ldots,g$.  Thus, it makes sense to
define $\theta:\Dbar(0,1)\rightarrow U_1\times\cdots\times U_g$ by
$$\theta(z) = (h_{1,\ell}(kz), \ldots, h_{g,\ell}(kz));$$
Then for all $n\geq 0$, we have
$$\theta(n) = \Phi^{nk+\ell}(\alpha),$$
because for each $i=1,\ldots,g$, we have $k_i|k$, and therefore
$$h_{i,\ell}(nk)=h_{i,\ell}(nk + h_{i,\ell}^{-1}(f_i^{\ell}(x_i)))
=f_i^{nk}(f_i^{\ell}(x_i))=f_i^{nk+\ell}(x_i).$$

Given any polynomial $F$ vanishing on $V$,
the composition
$F\circ \theta$ is a convergent
power series on $\Dbar(0,1)$
that vanishes at all integers in $\cN$.
However, a nonzero convergent power series can have only
finitely many zeros in $\Dbar(0,1)$;
see, for example, \cite[Section~6.2.1]{Robert}.  Thus,
$F\circ \theta$ is identically zero.  Therefore,
$$F(\Phi^{nk+\ell}(\alpha))=F(\theta(n))=0$$
for all $n\geq 0$, not just $n\in\cN$.
This is true for all such $F$, and therefore
$\cO_{\Phi^k}(\Phi^{\ell}(\alpha))\subseteq V(\Cp)$.

The conclusion of Theorem~\ref{indifferent}
now follows, because  $\cO$ is the finite union of the orbits
$\cO_{\Phi^k}(\Phi^{\ell}(\alpha))$ for $0\le \ell\le k-1$.
\end{proof}

As an immediate corollary, we have the following result, which proves
Conjecture~\ref{dynamical M-L} in the case that $\Phi$ is defined over
$\oQ$ and there is a nonarchimedean place
$v$ with the following property:
for each $i$, the rational
function $f_i$
has good reduction at $v$, and $\cO_{f_i}(x_i)$ avoids all $v$-adic
attracting periodic points.

\begin{theorem}
\label{avoiding attracting points}
Let $V$ be
a subvariety of $\left(\bP^1\right)^g$ defined over $\Cp$,
let $f_1,\dots, f_g\in\Cp(t)$ be rational functions
  of good reduction on
  $\bP^1$, and let $\Phi$ denote the coordinatewise action of
  $(f_1,\dots,f_g)$ on $\left(\bP^1\right)^g$. 
  Let $\cO$ be the $\Phi$-orbit of a point
  $\alpha=(x_1,\ldots,x_g)\in (\PCp)^g$,
  and suppose that for
  each $i$, the orbit $\cO_{f_i}(x_i)$ does not intersect the residue
  class of any attracting $f_i$-periodic point.
  Then $V(\Cp)\cap\cO$ is a union
of at most finitely many orbits of the form
$\{\Phi^{nk+\ell}(\alpha)\}_{n\geq 0}$ for
nonnegative integers $k$ and $\ell$.
\end{theorem}

\begin{proof}
For each $i$, the reduction $r_p(x_i)\in\bP^1(\overline{\mathbb{F}_p})$ is
preperiodic under the reduced map $(f_i)_p$.  Replacing
$\alpha$ by $\Phi^m(\alpha)$ for some $m\geq 0$,
and replacing $\Phi$ by $\Phi^j$ for some $j\geq 1$, then,
we may assume that for each $i$, the residue class $U_i$
of $x_i$ is mapped to itself by $f_i$.  By hypothesis,
there are no attracting periodic points in those
residue classes; thus, by \cite[Proposition~4.32]{Riv}
(for example), $U_i$ is a quasiperiodicity disk for $f_i$.
Theorem~\ref{indifferent} now yields the desired conclusion.
\end{proof}

\section{Preliminary results on intersection theory}
\label{intersection theory}

In this section we prove the following result
on intersection theory for
arithmetic surfaces.  It will be used in the proofs of our main
results in Section~\ref{sect:dml}.

\begin{theorem}\label{same v}
Let $K$ be a number field,
and let $\varphi: \bP^1_K \lra \bP^1_K$ be a
morphism defined over $K$ of degree at least $2$
that is not conjugate to a map of
the form $t \mapsto t^n$ for any integer $n$.
Suppose
$\varphi$ does not have any superattracting periodic points
other than exceptional points.

  Let $\alpha,\beta\in\bP^1(K)$ be points that are not
  preperiodic for $\varphi$.  Suppose that there is a curve $C
  \subseteq \bP^1 \times \bP^1$ such that there are infinitely many
  integers $k\geq 0$ for which $(\varphi^k(\alpha),\varphi^k(\beta)) \in C(K)$.
  Then there are infinitely many finite places $v$ of $K$ such that
  $\varphi$ has good reduction at $v$ and such that
  for some integer $n\geq 1$, the points
  $\varphi^n(\alpha)$ and $\varphi^n(\beta)$ are
  in the same residue class at the place $v$; i.e.,
  $r_v(\varphi^n(\alpha))=r_v(\varphi^n(\beta))$.
\end{theorem}
The condition on superattracting points is equivalent to stipulating
that the nonexceptional critical points of $\varphi$ are not periodic.
Note that $\varphi$ has at most one exceptional point, since it is not
conjugate to $t\mapsto t^n$.

By \cite[Proposition~4.2]{MorSil2}, $\varphi$ has good reduction
at all but finitely many places $v$ of $K$.
(See Section~\ref{notation} for a discussion of good reduction
and the reduction map $r_v$.)
Thus, the content
of Theorem~\ref{same v} is the common reduction of $\varphi^n(\alpha)$
and $\varphi^n(\beta)$.

Before proving the Theorem, we set some notation.  Let $V$ be a
variety over a number field $K$, and let $\cV$ be a model for $V$ over
the ring of integers $\fo_K$ of $K$.  Let $S$ be a finite set of
places of $K$ that contains all of the archimedean places of $K$, and
let $Z$ be an effective divisor on $V$.  We say that a point $\gamma$
on $V$ is {\em $S$-integral for $Z$} if the Zariski closure of
$\gamma$ does not meet the Zariski closure of $\Supp Z$ in $\cV$ at
any fibres of $\cV$ outside of $S$.

More specifically, let $\cV$ be the model
$\bP^1_{\fo_K}\times\bP^1_{\fo_K}$ for $\bP^1_K\times\bP^1_K$ that
comes from the isomorphism between $\bP^1_K$ and the generic fibre of
$\bP^1_{\fo_K}$ we chose in Section~\ref{notation}.  We will say that
a point $Q$ on $\bP^1_K \times \bP^1_K$ is \emph{$S$-integral} for a
divisor $Z$ if it is $S$-integral for $Z$ with respect to
$\cV$.

Let 
$$\Phi : \bP^1 \times \bP^1 \lra \bP^1 \times \bP^1$$
be the map
$\Phi = \varphi \times \varphi$, and let $\Delta$ denote the diagonal
divisor on $\bP^1 \times \bP^1$.
We will need the following proposition.

\begin{proposition}\label{integral}
  Let $\varphi: \bP_K^1 \lra \bP_K^1$ be a rational map of degree $d >
  1$ that has no periodic critical points.  Let $\alpha$ and $\beta$
  be points in $\bP^1(K)$ that are not preperiodic for $\varphi$, and
  let $S$ be a finite set of places of $K$ that contains all of
  the archimedean places of $K$.  Let $C$ be a curve in $\bP^1 \times
  \bP^1$.  Then there are at most finitely many
  integers $k\geq 0$ that satisfy both of
  the following conditions:
\begin{enumerate}
\item $\Phi^k(\alpha, \beta) \in C$; and
\item $\Phi^k(\alpha, \beta)$ is $S$-integral for $\Delta.$
\end{enumerate}
\end{proposition}

For any nonconstant morphism $h:\bP^1\to \bP^1$
and any point $x$ on $\bP^1$,
we will denote 
the ramification index of $x$ over $h(x)$ by $e(x/h(x))$.

We will need the following result about ramification.
\begin{lemma}
\label{needed ramification}
Let $h:\bP^1\lra \bP^1$ be a nonconstant morphism defined over a field
$K$ of characteristic $0$, and let $H:=(h,h)$ its action
coordinatewise on $\bP^1\times\bP^1$. Then:
\begin{enumerate}
\item
For any point $(P,Q)\in\bP^1(K)\times \bP^1(K)$, the
multiplicity of $\Delta_H:=H^*(\Delta)$ at $(P,Q)$ is at most
$\max_{x\in\bP^1} e(x/h(x))$.
\item
Each irreducible component
of $\Delta_H$ has multiplicity one.
\end{enumerate}
\end{lemma}

\begin{proof}
By performing the same change of coordinates on both copies of $\bP^1$, we may assume that the point at infinity is not among the points $P,h(P),Q,h(Q)$. Hence, let $t_0,u_0\in K$ such that $P=[t_0:1]$, and $Q=[u_0:1]$.
Then $h$ has a local power series expansion (see
\cite[II.2]{Shaf}) in a neighborhood of $P$ as
$h(t) = a_0 + \sum_{i=e_1}^\infty a_i (t-t_0)^i$
and in a neighborhood of $Q$ as
$h(u) = b_0 + \sum_{i=e_2}^\infty b_i (u-u_0)^i$, where
$a_i, b_i\in K$, 
and $e_1\ge 1$ and $e_2\ge 1$ are the ramification indices of
$h$ at $P$ and $Q$, respectively. Clearly, $(P,Q)\in \Delta_H$ if and only if $a_0=b_0$.
Thus, we may assume $a_0=b_0$, and so, near $(P,Q)$, the
subvariety $\Delta_H$ is defined by the equation
$$ \sum_{i=e_1}^\infty a_i (t-t_0)^i - \sum_{i=e_2}^\infty b_i (u-u_0)^i=0.$$
The multiplicity of $(P,Q)$ as a point on $\Delta_H$ is therefore given by
$\min(e_1, e_2)$ (see \cite[IV.1]{Shaf}); 
since $e_1,e_2 \leq \max_{x\in\bP^1}e(x/h(x))$, statement~(i) follows.

Moreover, $\Delta_H$ has
multiplicity more than one at $(P,Q)$ only if $h$ is ramified at both
$P$ and $Q$.
Because $h$ is ramified at only finitely many points of $\bP^1$
(note that $\car(K)=0$),
there are at most finitely many points
$(P,Q)\in\bP^1\times\bP^1$ at
which $\Delta_H$ has multiplicity larger than one,
proving statement~(ii).
\end{proof}

We set more notation, as follows.
For each $n \geq 0$, let $X_n$ be the divisor $(\Phi^n)^{*} (\Delta)$.
Note that $\Delta \subseteq X_n$ for each $n$. Therefore, more
generally, for each $0\leq m < n$, we have $X_m \subseteq X_n$.
Let $Y_0:= X_0 = \Delta$, and for $n\ge 1$, let
$$ Y_n =  (\Phi^n)^{*} (\Delta ) - 
(\Phi^{n-1})^{*} (\Delta )  .$$ 
Then we have
$$
X_n = \bigcup_{i=0}^n Y_i.
$$
Note that $Y_n$ is nonempty because $\deg(\Phi)>1$.
Furthermore, by Lemma~\ref{needed ramification}, 
each irreducible component of $X_n $, and hence of $Y_n$,
has multiplicity one.

We also have the following important result, giving a uniform
bound for the ramification of $\varphi^n$.

\begin{lemma}\label{ramlem}
Let $\varphi:\bP^1\lra\bP^1$ be a map which has no periodic critical points. Let $Q_1,\dots,Q_m$ be the ramification points of $\varphi$.  Then for
any $n$ and any point $P$ on $\bP^1$,
\begin{equation}
\label{eq:ramlem}
e(P/\varphi^n(P)) \leq \prod_{i=1}^m e(Q_i/\varphi(Q_i)).
\end{equation}
\end{lemma}
\begin{proof}
For each integer $i=1,\ldots, m$, there is at most one $j\geq 0$ such
that $\varphi^j(P) = Q_i$, since none of the $Q_i$ are periodic.
Meanwhile,
$e(\varphi^j(P)/\varphi^{j+1}(P))$ equals $1$
for all $j\geq 0$ such that $\varphi^j(P)$ is
not a ramification point.  Thus,
\[
e(P/\varphi^n(P)) = \prod_{j=0}^{n-1}
e(\varphi^j(P)/\varphi^{j+1}(P))
\leq \prod_{i=1}^m e(Q_i/\varphi(Q_i)).
\qedhere
\]
\end{proof}

Combining Lemmas~\ref{needed ramification} and \ref{ramlem}
gives the following result.
\begin{lemma}\label{mult}
Under the hypothesis of Lemma~\ref{ramlem} and with the above notation for $X_n$ and $Y_n$, there is a constant $M\geq 0$ such that for any point
$Q \in \bP^1(\bK) \times \bP^1(\bK)$,
at most $M$ of the $Y_n$ contain $Q$.  
\end{lemma}

\begin{proof}
Let $M$ be the quantity on the right hand side of \eqref{eq:ramlem}.
If a point $Q$ is contained in $M+1$ different $Y_i$, then the
multiplicity of $Q$ on $X_n$ is at least $M+1$ for some
large enough $n$.
Lemma~\ref{ramlem} and
part~(i) of Lemma~\ref{needed ramification} 
now give a contradiction.
\end{proof}

We are now ready to prove  Proposition~\ref{integral}.
\begin{proof}[Proof of Proposition~\ref{integral}]
  Enlarge $S$ if necessary to contain not only all
  archimedean places of $K$ but also all
  places of bad reduction for $\Phi$.
  Fix an integer $n \geq 2M$, where $M$ is as in
  Lemma~\ref{mult}.

  Given an irreducible curve $E$ in
  $\bP^1 \times \bP^1$ that does not map to a point under either of the
  projection maps on $\bP^1 \times \bP^1$, we claim that
  $E$ intersects $X_n$ in at least three distinct points.
  Indeed, $E$ must meet $Y_m$
  in at least one point for all $m \geq 0$.
  However, for each point $Q$,
  at most $M$ of the $Y_m$ (for $0\le m\le n$) contain $Q$.
  Hence, $E$ must intersect $X_n$ in at least three
  distinct points, as desired.

  Now suppose that there are infinitely many $k$ (and hence infinitely
  many $k>n$) such that $\Phi^k(\alpha,\beta) \in C$ and
  $\Phi^k(\alpha,\beta)$ is $S$-integral for $\Delta$.  For each such
  $k>n$, then, $\Phi^{k - n}(\alpha,\beta)$ is $S$-integral for $X_n$.
  Then there is a $K$-irreducible curve $Z$ in $(\Phi^n)^{-1}(C)$ such
  that there are infinitely many $m$ for which $\Phi^m(\alpha,\beta)
  \in Z$ and $\Phi^m(\alpha,\beta)$ is $S$-integral for $X_n$.
  Because $\alpha$ and $\beta$ are not preperiodic, $Z$ does not
  project to a single point on either of the two coordinates of
  $\bP^1\times \bP^1$.  In addition, $Z$ contains infinitely many
  $K$-rational points $\Phi^m(\alpha,\beta)$; because it is also
  irreducible over $K$, it is in fact geometrically irreducible (note
  that if a component of $Z$ defined over a finite extension of $K$
  has infinitely many $K$-rational points, then it is in fact defined
  over $K$.)

  Thus, by our claim, $X_n$ meets $Z$
  in at least three points.  However,
  $Z$ contains infinitely many points $\Phi^m(\alpha,\beta)$
  that are $S$-integral for $X_n$; this
  is impossible, by Siegel's theorem on integral points.
\end{proof}

We treat the case of polynomials (which {\em do} have a periodic critical
point) slightly differently.  For a polynomial $f(t) = \sum_{i=0}^d
a_i t^i$ with $a_d \not = 0$, we define its homogenization $F(t,u)$ by
$F(t,u) = \sum_{i=0}^d a_i t^i u^{d-i}$.  We then define
$\Phi_f:\bP^2 \lra \bP^2$ by
$$ \Phi_f([x:y:z]) = [F(x,z):F(y,z):z^d].$$
Let $D$ be the divisor on $\bP^2$ consisting of all points $[x:y:z]$
such that $x = y$.  The divisor $D$ will play the same role here that
the diagonal $\Delta$ played on $\bP^1 \times \bP^1$.  We let 
$$ A_n =  (\Phi_f^n)^* (D),
\quad \text{and} \quad
B_n = (\Phi_f^n)^* (D) - (\Phi_f^{n-1})^* (D).$$
Then 
$$ A_n = \bigcup_{i=0}^n B_i.$$
Let $\cW$ be the model $\bP^2_{\fo_K}$ for $\bP^2_K$.  We will say
that a point $Q$ on $\bP^2_K$ is $S$-integral for a divisor
$Z$ if it is $S$-integral for $Z$ with respect to
$\cW$.

\begin{proposition}\label{poly}
  Let $f \in K[t]$ be a polynomial with no periodic critical points
  other than the point at infinity.  Let $\alpha$ and $\beta$ be
  points in $\bA^1(K)$ that are not preperiodic for $f$, and let $S$ be
  a finite set of places of $K$ that contains all of the archimedean
  places.  Let $C$ be a curve in $\bP^2$.  Then there are at most
  finitely many $k$ that satisfy both of the following conditions:
\begin{enumerate}
\item $\Phi_f^k([\alpha: \beta:1]) \in C$; and
\item $\Phi_f^k([\alpha: \beta:1])$ is $S$-integral for $D$.
\end{enumerate}
\end{proposition}
\begin{proof}
The proof is almost identical to the proof of Proposition
\ref{integral}.  Note that if $[x:y:0]$ lies on $A_n$, then $[x:y:0]$
must be in the inverse image of $[1:1:0]$ under $\Phi_f$, which is
equivalent to saying that $x^{d^n} = y^{d^n}$.  There are exactly
$d^n$ such points, and $d^n$ is also the degree of $A_n$; so each
point of the form $[x:y:0]$ must have multiplicity one on $A_n$ (and
hence on $B_n$) for any $n$.  Then, as in Proposition \ref{integral},
we can bound the multiplicity of any point $[x:y:1]$ on $A_n$ by 
$$\big( \prod_{i=1}^m e(Q_i/f(Q_i)) \big)$$
where the $Q_i$ are the
ramification points of $f$ other than infinity.  Thus, again, if there
are infinitely many points on $C$ that are $S$-integral for $D$, then
for any $n$, there are infinitely many points on some irreducible
curve $E$ in $(\Phi_f^n)^{-1}(C)$ that are $S$-integral for $A_n$.
When $n$ is at least 
$2\cdot \prod_{i=1}^m e(Q_i/f(Q_i))$, such a curve
$E$ must meet $A_n$ in at least three distinct points, which gives us
a contradiction by Siegel's theorem.
\end{proof}

We are now ready to prove Theorem~\ref{same v}.

\begin{proof}[Proof of Theorem~\ref{same v}]
  If $\varphi: \bP^1 \lra \bP^1$ has no periodic
  critical points, Proposition~\ref{integral} implies that for any
  finite set $S$ of places of $K$, there are only finitely many $n$ such
  that $\varphi^n(\alpha)$ does not meet $\varphi^n(\beta)$ at any $v$
  outside of $S$.  Thus, there must be infinitely many places $v$ such
  that $r_v(\varphi^n(\alpha)) = r_v(\varphi^n(\beta))$ for some $n\in\N$.

  On the other hand, if
  $\varphi$ has an exceptional point, then after changing coordinates,
  we have $\varphi = f$ for some polynomial $f$ (note that $\varphi$ does not have two exceptional points, as it is not conjugate to a map of the form $t\mapsto t^n$). Furthermore, since $\varphi$ has
  no non-exceptional periodic critical points, it follows that $f$ has 
  no periodic critical points save the point at infinity.  By
  Proposition~\ref{poly}, for any finite set $S$ of places of $K$,
  there are at most finitely many $n$ such that
  $f^n(\alpha) - f^n(\beta)$
  is an $S$-unit.  Thus, there are infinitely many
  places $v$ such that $r_v(f^n(\alpha)) = r_v( f^n(\beta))$ for some $n\in\N$.
\end{proof}

\section{Dynamical Mordell-Lang  for curves}
\label{sect:dml}

Using Theorem~\ref{same v} we can prove a dynamical Mordell-Lang
statement for curves embedded in $\bP^1\times \bP^1$.
\begin{theorem}
\label{plane curves}
Let $C\subset \bP^1\times\bP^1$ be a curve defined over $\oQ$, and let
$\Phi:=(\varphi,\varphi)$ act on $\bP^1\times\bP^1$, where
$\varphi\in\oQ(t)$ is a rational function with no superattracting periodic
points other than exceptional points. Let $\cO$ be the
$\Phi$-orbit of a point $(x,y)\in\left(\bP^1\times
\bP^1\right)(\oQ)$. Then $C(\oQ)\cap\cO$ is a union of
at most finitely many
orbits of the form $\{\Phi^{nk+\ell}(x,y)\}_{n\ge 0}$ for
$k,\ell\in\bN$.
\end{theorem}

\begin{proof}
  When $\deg (\varphi) = 1$, the result follows immediately from work
  of Denis \cite{Denis-dynamical} and Bell \cite{Bell}, since in this
  case $\Phi$ induces an automorphism of $\bA^2$. Hence, we may assume
  $\deg(\varphi)\ge 2$.

If $\varphi$ has two exceptional points,
then $\varphi$ is conjugate to the map $t\mapsto t^n$, for
some $n\in\bZ$. Then our result follows from
\cite[Theorem 1.8]{GT-newlog},
as $\Phi$ induces an endomorphism of
$\bG_m^2$.  Thus, we may assume that $\varphi$ has at most
one exceptional point, and no other periodic critical points.
  
  We may assume that $C$ is irreducible, and that $C(\oQ)\cap\cO$ is
  infinite.  We may also assume that neither $x$ nor $y$ is
  $\varphi$-preperiodic, because in that case the projection of $C$ to
  one of the two coordinates of $\bP^1\times\bP^1$ consists of a
  single point (which would be a $\varphi$-periodic point), and the
  conclusion of Theorem~\ref{plane curves} would be immediate.

Let $K$ be a number field over which $\varphi$, $C$, and $(x,y)$ are
defined.  By the previous paragraph, the
hypotheses of Theorem~\ref{same v} hold for $(\alpha,\beta)=(x,y)$.
Thus, there are 
infinitely many nonarchimedean places $v$ of $K$
at which $\varphi$ has good reduction and such that
$r_v(\varphi^{n}(x))=r_v(\varphi^{n}(y))$ for some integer
$n\geq 1$.  Fix such a place $v$.

Let $p\in\N$ be the prime number lying in the maximal ideal
of the nonarchimedean place $v$, fix an embedding of $K$
into $\Cp$ respecting $v$, and let $U$ denote the
residue class of $\bP^1(\Cp)$ containing
$\varphi^n(x)$ and $\varphi^n(y)$.  Since $\varphi$
has good reduction, every iterate $\varphi^{n+k}(U)$ is a
residue class, and it contains both $\varphi^{n+k}(x)$
and $\varphi^{n+k}(y)$.  If no such residue class
contains an attracting periodic point, 
then  our desired conclusion
is immediate from Theorem~\ref{avoiding attracting points}.

The remaining case is that some residue class
$\varphi^{n+k}(U)$ contains an attracting periodic point,
which must therefore attract the orbits of both $x$ and $y$.
The Theorem now follows from 
\cite[Theorem~1.3]{GT-newlog}.
\end{proof}

We can now prove Theorem~\ref{curves}
as a consequence of Theorem~\ref{plane curves}.


\begin{proof}[Proof of Theorem~\ref{curves}]
We may assume that $C$ is irreducible, and that $C(\oQ)\cap \cO$ is
infinite. It suffices to prove that $C$ is $\Phi$-periodic. Indeed, if
$\Phi^k(C) = C$, then for each $\ell\in\{0,\dots,k-1\}$, the intersection
of $C$ with $\cO_{\Phi^k}(\Phi^{\ell}(\alpha))$
either is empty or else consists
of all $\Phi^{kn+\ell}(\alpha)$, for some $n$ sufficiently large. Either way,
the conclusion of Theorem~\ref{curves} holds.

We argue by induction on $g$. The case $g=1$ is obvious, while the
case $g=2$ is proved in Theorem~\ref{plane curves}.
Assuming
Theorem~\ref{curves} for some $g\ge 2$, we will now prove it
for $g+1$. We may assume that $C$ projects dominantly onto each of the
coordinates of $\left(\bP^1\right)^{g+1}$; otherwise, we may view $C$
as a curve in $\left(\bP^1\right)^g$, and apply the inductive
hypothesis.
We may also assume that no $x_i$ is preperiodic,
lest $C$ should fail to project dominantly
on the $i^{th}$ coordinate.

Let $\pi_1:\left(\bP^1\right)^{g+1}\to\left(\bP^1\right)^g$
be the projection onto the first $g$ coordinates,
let $C_1:=\pi_1(C)$, and let $\cO_1:=\pi_1(\cO)$.
By our assumptions, $C_1$ is an
irreducible curve that has an infinite intersection with $\cO_1$. By
the inductive hypothesis, $C_1$ is periodic under the
coordinatewise action of $\varphi$ on the first $g$ coordinates of
$\left(\bP^1\right)^{g+1}$.

Similarly, let $C_2$ be the projection of $C$ on the last $g$
coordinates of $\left(\bP^1\right)^{g+1}$. By the same argument,
$C_2$ is periodic under the coordinatewise action of $\varphi$
on the last $g$ coordinates of $\left(\bP^1\right)^{g+1}$.

Thus, $C$ is $\Phi$-preperiodic,
because it is 
an irreducible component of the
one-dimensional variety
$\left(C_1\times \bP^1\right) \cap \left(\bP^1 \times C_2\right)$,
and because both $C_1\times \bP^1$ and $\bP^1\times
C_2$ are $\Phi$-periodic.

\begin{claim}
\label{preperiodic but not periodic}
Let $X$ be a variety,
let $\alpha\in X(\Kbar)$,
let $\Phi:X\lra X$ be a morphism, 
and let $C\subset X$ be an irreducible curve that has infinite intersection
with the orbit $\cO_{\Phi}(\alpha)$.
If $C$ is $\Phi$-preperiodic, then $C$ is $\Phi$-periodic.
\end{claim}

\begin{proof}[Proof of Claim~\ref{preperiodic but not periodic}.]
Assume $C$ is not periodic.
Because $C$ is preperiodic, there exist 
$k_0,n_0\ge 1$ such that $\Phi^{n_0}(C)$ is periodic of period $k_0$.
Let $k:=n_0k_0$, and let $C':=\Phi^k(C)$, which is fixed by $\Phi^k$.
Then $C\ne C'$, since $C$ is not periodic.
Because  $C$ and $C'$ are irreducible curves, it follows that
\begin{equation}
\label{equ 1}
C\cap C'\text{ is finite.} 
\end{equation}

On the other hand,
there exists $\ell\in\{0,\dots,k-1\}$ such that
$C\cap\cO_{\Phi^k}(\Phi^{\ell}(\alpha))$ is infinite,
because $C\cap\cO_{\Phi}(\alpha)$ is infinite.
Let $n_1\in \N$ be the smallest nonnegative integer $n$ such that
$\Phi^{nk+\ell}(\alpha)\in C$. Since $C'=\Phi^k(C)$ is fixed by
$\Phi^k$, we conclude that $\Phi^{nk+\ell}(\alpha)\in C'$ for each
$n\ge n_1+1$. Therefore
\begin{equation}
\label{equ 2}
C\cap\cO_{\Phi^k}(\Phi^{\ell}(\alpha))\cap C'\text{ is infinite.}
\end{equation}
Statements \eqref{equ 1} and \eqref{equ 2} are contradictory,
proving the claim.
\end{proof}
An application of Claim~\ref{preperiodic but not periodic}
with $X=(\bP^1)^{g+1}$ now completes the proof of Theorem~\ref{curves}.
\end{proof}

\section{Quadratic polynomials}
\label{sect: quadratic}

In this Section, we will prove Theorems \ref{quadratic} and
\ref{rationals}.  We will continue to work with the same reduction
maps $r_v: \bP^1(K) \lra \bP^1(k_v)$ as in Section~\ref{notation},
where $v$ is a finite place of $K$.  We begin with a lemma derived
from work of Silverman \cite{SilSiegel}.

\begin{lemma}\label{simple siegel}
  Let $\varphi: \bP^1 \lra \bP^1$ be a morphism of degree greater than
  one, let $\alpha \in \bP^1(K)$ be a point that is not preperiodic for
  $\varphi$, and let $\beta \in \bP^1(K)$ be a nonexceptional point for $\varphi$.
  Then there
  are infinitely many $v$ such that there is some positive integer $n$
  for which $r_v(\varphi^n(\alpha)) = r_v(\beta)$.  
\end{lemma}
\begin{proof}
  Suppose there were only finitely many such $v$; let $S$ be the
  set of all such $v$, together with all the archimedean places.
  We may choose coordinates $[x:y]$ for $\bP^1_K$ such that $\beta$ is
  the point $[1:0]$.  Since $[1:0]$ is not exceptional for
  $\varphi$, we see that $\varphi^2$ is not a polynomial with respect
  to this coordinate system.  Therefore,
  by \cite[Theorem 2.2]{SilSiegel},
  there are at most finitely many $n$ such
  that $\varphi^n(\alpha) =  [t:1]$ for $t \in \fo_S$,
  where $\fo_S$ is the ring of $S$-integers in $K$.  Hence, for all but
  finitely many integers $n\geq 0$, there is some $v \notin S$ such that
  $r_v(\varphi^n(\alpha)) = r_v(\beta)$; but this contradicts
  our original supposition.
\end{proof}

Recall that if $f$ has good reduction at a finite place $v$
of $K$, we write $f_v$ for the reduction of $f$ at $v$.

\begin{lemma}\label{sil}
  Let $\varphi: \bP^1 \lra \bP^1$ be a morphism of degree greater than
  one, and let $\alpha \in K$ be a point that is not periodic for
  $\varphi$.  Then there are infinitely many places $v$ of good
  reduction for $\varphi$ such that $r_v(\alpha)$ is not periodic for
  $\varphi_v$.
\end{lemma}
\begin{proof}
If $\alpha$ is $\varphi$-preperiodic but not
periodic, then the $\varphi$-orbit $\cO_{\varphi}(\alpha)$ is finite.
Hence,  the reduction map $r_v$ is injective on $\cO_{\varphi}(\alpha)$
for all but finitely many places $v$, and Lemma~\ref{sil} holds in
this case.

Thus, we may assume that $\alpha$ is not preperiodic.  After passing
to a finite extension $L$ of $K$, we may also assume that $\varphi$
has a nonexceptional fixed point $\beta$. We extend our isomorphism between $\bP^1_K$ and the generic fibre of $\bP^1_{\fo_K}$ to an isomorphism from $\bP^1_L$ to the generic
fibre of $\bP^1_{\fo_L}$; and for each place $w|v$ of $L$, we obtain
reduction maps $r_w: \bP^1(L) \lra \bP^1(\ell_w)$, where $\ell_w$ is
the residue field at $w$.  For each such $w |v$, we have $r_v(\gamma)
= r_w(\gamma)$ for any $\gamma \in \bP^1(K)$.  By Lemma~\ref{simple
  siegel}, there are infinitely many places $w$ such that there is
some $n$ for which $r_w(\varphi^n(\alpha)) = r_w(\beta)$.  When $w|v$
for $v$ a place of good reduction for $\varphi$, this means that
$r_v(\varphi^m(\alpha)) = r_v(\varphi^n(\alpha))=r_w(\beta)$ for all
$m \geq n$, since $\beta$ is fixed by $\varphi$.  At all but finitely
many of these $v$, we have $r_v(\alpha) \not= r_w(\beta)$, which means
that there is no positive integer $m$ such that
$r_v(\varphi^m(\alpha)) = r_v(\alpha)$, as desired.
\end{proof}

We also need the following result for quadratic polynomials.
\begin{proposition}\label{quad 1}
  Let $K$ be a number field, and let $f \in K[t]$ be a quadratic
  polynomial with no periodic critical points other than the point at
  infinity.  Then there are infinitely many finite places $v$ of $K$
  such that $|f'(z)|_v = 1$ for each $z\in K$ such that $|z|_v\leq 1$
  and $r_v(z)$ is $f_v$-periodic.
\end{proposition}
\begin{proof}
  Since $f$ is a quadratic polynomial, it only has one
  critical point $\alpha$ other than the point at infinity.
  By Lemma~\ref{sil} and because
  $\alpha$ is not periodic, there are infinitely many places $v$
  of good reduction for $f$ such that $r_v(\alpha)$ is not
  $f_v$-periodic, and such that
  $|\alpha|_v\leq 1$ and $|2|_v=1$.  (The last two
  conditions may be added because each excludes only
  finitely many $v$.)  In particular,
  $|f'(z)|_v = |z-\alpha|_v$ for any $z\in K$.

  Hence, for any such $v$, and for any
  $z\in K$ as in the hypotheses, we have
  $r_v(z)\neq r_v(\alpha)$, since $r_v(z)$ is
  periodic but $r_v(\alpha)$  is not.
  Thus, $|f'(z)|_v=|z-\alpha|_v=1$.
\end{proof}

We are now ready to prove Theorems~\ref{quadratic},
\ref{rationals}, and \ref{rationals-2}.

\begin{proof}[Proof of Theorem~\ref{quadratic}.]
Let $K$ be a number field such that $V$ is defined over
$K$, the polynomial $f$ is in $K[t]$,  and $x_1,\dots,x_g$ are all in
$K$.  
 
Using Proposition~\ref{quad 1}, we may choose a place $v$ of $K$
  such that
\begin{enumerate}
\item[$(a)$] $v$ is a place of good reduction for $f$;
\item[$(b)$] $|x_i|_v\le 1$, for each $i=1,\dots,g$;
\item[$(c)$] $|f'(z)|_v = 1$ for all $z$ such that
  $|z|_v\leq 1$ and $r_v(z)$ is $f_v$-periodic.
\end{enumerate}
Indeed, conditions $(a)$ and $(b)$ are satisfied at all but finitely many
places $v$, while condition $(c)$ is satisfied at infinitely many
places. 
Because $f$ is a polynomial, conditions $(a)$ and $(b)$ together
imply that $|f^n(x_i)|_v\leq 1$ for all $i=1,\dots,g$ and $n\geq 0$.
Meanwhile, condition $(c)$
implies that $f$ has no attracting periodic points at $v$.
The desired conclusion now
follows from Theorem~\ref{avoiding attracting points}.
\end{proof}

\begin{proof}[Proof of Theorem~\ref{rationals}.]
  After changing coordinates, we assume that $f(t)=t^2 + c$
  for some $c\in\bQ$.
  Thus, $0$ is the only finite critical
  point of $f$.  If $c\not\in\bZ$, then there is some $p$ such
  that $|c|_p > 1$.  But then $|f^n(0)|_p \to \infty$, so $0$ cannot
  be periodic.  Similarly, if $c$ is an integer other than $0$, $-1$
  or $-2$, then we have $|f^n(0)|_{\infty} \to \infty$, so $0$ cannot be
  periodic. If $c=-2$, then $0$ is only $f$-preperiodic, but not
  $f$-periodic.  In all the above cases,  the hypotheses of
  Theorem~\ref{quadratic} are met, and our proof is done.
  If $c=0$, then $f(t)=t^2$ is
  an endomorphism of $\bG_m^g$, and
  thus our result follows from \cite[Theorem 1.8]{GT-newlog}.

  We are left with the case that $f(t)=t^2-1$.  
  As in the proof of Theorem~\ref{curves}, we may assume
  (via induction on $g$) that no $x_i$ is preperiodic;
  in particular, all $x_i$ and $f(x_i)$ are nonzero.
If $f^2(z) = 0$, then
  either $z=0$, or $z = \pm \sqrt{2}$.  Bearing this fact in
  mind, we note that there are infinitely many
  primes $p$ such that $2$ is not a quadratic residue modulo $p$.
  Thus, we may choose an odd prime $p$ such that
  each $x_i$ and $f(x_i)$ is a $p$-adic unit,
  and such that $2$ is not a quadratic residue modulo $p$.
  Then there is no positive integer $n$ such that $f^n(x_i)$ is
  in the same residue class as $0$ modulo $p$ for any $i$.
  Therefore, $|f'(f^n(x_i))|_p=1$ for all $n$, and hence
  $f^n(x_i)$ never lies in the same residue class as
  an attracting periodic point.
  Theorem~\ref{rationals} now
  follows from Theorem~\ref{avoiding attracting points}.
\end{proof}

\begin{proof}[Proof of Theorem~\ref{rationals-2}.]
As before, we may assume that no $x_j$ is preperiodic for $f_j$.
By \cite[Theorem 1.2(iii)]{Jones},
for each $f_j$ that is not equal to $t^2 - 1$, the set
of primes $p$ such that there is an $n$ for which
$f_j^n(x_j) \equiv 0 \pmod{p}$ has Dirichlet density zero.
Meanwhile, as noted
in the proof of Theorem~\ref{rationals},
the density of primes $p$ such that $-2$ is a square modulo $p$
is $1/2$, and therefore the set of primes $p$ for which
there are an $n$ and and $j$ satisfying $f_j(t)=t^2-1$
and $f_j^n(x_j) \equiv 0 \pmod{p}$ must have (upper) density
at most $1/2$.  Hence, the set of primes $p$
such that $f^n_j(x_j) \not\equiv 0 \pmod{p}$ for all $n$ and
all $j = 1,\dots, g$ has (lower) density at least $1/2$.
Choosing such a prime $p$, we see that $f_j^n(x_j)$
never lies in the same residue class as an attracting periodic point
for any $n$ and any $j = 1, \dots, g$, and the result
follows from
Theorem~\ref{avoiding attracting points}.
\end{proof}

\section{Extensions over the field of complex numbers}
\label{extensions to C}

In this section we will use recent work of
Medvedev and Scanlon \cite{Medvedev/Scanlon}
to prove Theorem~\ref{polynomials, arbitrary curves over C}.
We begin with the following definitions.

\begin{definition}
\label{def:indecomp}
Let $K$ be a field, and let $\varphi\in K[t]$ be a
nonconstant polynomial.
We say that $\varphi$ is {\em indecomposable} if there are no
polynomials $\psi_1,\psi_2\in\overline{K}[t]$ of degree greater
than one such that $\varphi=\psi_1\circ\psi_2$.
\end{definition}

Generic polynomials of any positive degree are indecomposable.
This is obvious for
(all) polynomials of prime degree or degree one and
easy to prove
in degree at least $6$ (say by reducing to monic decompositions and
counting dimensions); but it can also be shown in
degree $4$.

\begin{definition}
\label{def:normform}
Let $K$ be a field, and let $f\in K[t]$ be a
polynomial of degree $m\geq 1$.
If $f$ is monic with trivial $t^{m-1}$ term,
we say that
$f$ is in {\em normal form}; that is,
$f$ is of the form
$$t^m + c_{m-2}t^{m-2} + \cdots + c_0.$$
In that case, we say that $f$ is of {\em type} $(a,b)$
if $a$ is the smallest nonnegative integer such that
$c_a\neq 0$, and $b$ is the largest positive integer
such that $f(t)=t^a u(t^b)$ for some polynomial
$u\in K[t]$.
\end{definition}

While we have introduced this definition of ``type'' to
aid our exposition, the accompanying notion of
normal form is not new.  In fact, as noted
in \cite[Equation~(2.1)]{Beardon-2},
if $\car K=0$ and
$f\in K[t]$ is a polynomial of degree $m\geq 2$,
and if
$K$ contains an $(m-1)$-st root of the leading coefficient,
then 
there is a linear polynomial $\mu\in K[t]$
such that $\mu^{-1}\circ f\circ\mu$
is in normal form.


\begin{definition}
For each positive integer $m$, define $D_m\in \mathbb{Z}[t]$
to be the unique polynomial of degree $m$ such that
$D_m(t+1/t)=t^m+1/t^m$.
\end{definition}

The usual Chebyshev polynomial
$T_m$ (satisfying $T_m(\cos(\theta))=\cos(m\theta)$)
is conjugate to $D_m$, since $D_m(2t)=2T_m(t)$.
However, $D_m$ is in normal form.

The following result is an immediate consequence of Theorem 3.149 in
\cite{Medvedev/Scanlon} (see also Section 3.2 in
\cite{Medvedev/Scanlon}).

\begin{theorem}[Medvedev, Scanlon]
\label{conj: periodic curves}
Let $K$ be a field of characteristic $0$, and let $\varphi\in K[t]$ be
a nonlinear indecomposable polynomial
which is not conjugate to $t^m$ or $D_m$ for any
positive integer $m$.
Assume that $\varphi$ is in normal form, of type $(a,b)$.

Let $\Phi$ denote the action of
$(\varphi,\varphi)$ on $\bA^2$. Let $C$ be a $\Phi$-periodic
irreducible plane curve defined over $K$.  Then $C$ is defined by one
of the following equations in the variables $(x,y)$ of the affine
plane:
\begin{itemize}
\item[(i)] $x = x_0$, for a $\varphi$-periodic point $x_0$; or
\item[(ii)] $y = y_0$, for a $\varphi$-periodic point $y_0$; or
\item[(iii)] $x = \zeta \varphi^r(y)$, for some $r \ge 0$; or
\item[(iv)] $y = \zeta \varphi^r(x)$, for some $r \ge 0$,
\end{itemize}
where $\zeta$ is a $d$-th root of unity,
where $d\mid b$ and $\gcd(d,a)=1$.
\end{theorem}

\begin{remark}
Note that if $b=1$ or $a=0$
in Theorem~\ref{conj: periodic curves},
then $d=1$, and hence $\zeta=1$.
\end{remark}

Using Theorem~\ref{conj: periodic curves},
we can prove the following result.
\begin{theorem}
\label{polynomials, plane curves over C}
Let $\varphi\in\bC[t]$ be an indecomposable polynomial
with no periodic superattracting points other than exceptional points,
and let $\Phi:=(\varphi,\varphi)$ be its
diagonal action on $\bA^2$. Let $\cO$ be the $\Phi$-orbit of a point
$(x_0,y_0)$ in $\bA^2(\bC)$, and let $C$ be a curve defined over
$\bC$. Then $C(\bC)\cap\cO$ is a union of at most finitely many orbits of the
form $\{\Phi^{nk+\ell}(x_0,y_0)\}_{n\ge 0}$ for nonnegative integers
$k$ and $\ell$.
\end{theorem}

We will need three more ingredients to prove
Theorem~\ref{polynomials, plane curves over C}.

\begin{proposition}
\label{twisted Laurent}
Fix integers $m,g\geq 1$,
let $\varphi\in\bC[t]$ be a polynomial which is a conjugate of either
$t^m$ or $D_m$, and let $\Phi$ be its coordinatewise
action on $\bA^g$.  Let $\cO$ be the
$\Phi$-orbit of a point $\alpha\in\bA^g(\bC)$, and let $V$ be an affine
subvariety of $\bA^g$ defined over $\bC$. Then $V(\bC)\cap\cO$ is a
union of at most finitely many orbits of the form
$\{\Phi^{nk+\ell}(P)\}_{n\ge 0}$
for nonnegative integers $k$ and $\ell$.
\end{proposition}

\begin{proof}
By hypothesis, there is a linear polynomial $h(t)\in\bC[t]$ such
that either $\varphi(h(t))=h(t^m)$ or $\varphi(h(t))=h(D_m(t))$.  In the
first case, let $k(t)=h(t)$, and in the second, let
$k(t)=h(t+1/t)$.  Then
$\varphi(k(t))=k(t^m)$ for some nonconstant rational function
$k(t)\in\bC(t)$.  Note that in either case, $k(\bC)\supseteq\bC$,
and the only possible poles of $k$ are at $0$ and $\infty$.

Let $\alpha:=(x_1,\ldots,x_g)$. For each $i\in\{1,\dots,g\}$,
pick $z_i\in\bC$ such that $k(z_i)=x_i$.  Then
$\cO_{\Phi}(\alpha)=
\big\{\big(k(z_1^{m^n}),\ldots,k(z_g^{m^n})\big) : n\geq 0\big\}$.
Let $W$ be the affine subvariety of $\bG_m^g$ defined by
the equations $f(k(t_1),\dots,k(t_g))=0$, where
$f$ ranges over a set of generators for the vanishing
ideal of $V$. (Note that $W$ is an algebraic subvariety of $\bG_m^g$
because $k$ has no poles on $\bG_m$.)

Let $\Psi$ be the endomorphism of $\bG_m^g$ given by
$\Psi(t_1,\dots,t_g)=(t_1^m,\dots,t_g^m)$. Then
$$\Phi^n(x_1,\dots,x_g)\in V(\bC)\text{ if and only if }
\Psi^n(z_1,\dots,z_g)\in W(\bC).$$
Thus,
Proposition~\ref{twisted Laurent} holds for $\Phi$ and $V$
because,
by \cite[Theorem $1.8$]{GT-newlog}, it holds
for $\Psi$ and $W$.
\end{proof}


\begin{proposition}
\label{needed conjugate reduction}
Let $E$ be a field of characteristic $0$,
and $K$ a function field of transcendence degree
$1$ over $E$. Let $\varphi\in K[t]$ be a polynomial of degree
$m\ge 2$ in normal form.
Assume that $\varphi$ is not conjugate to $t^m$ or $D_m$.
Then for all but finitely places $v$ of the function field $K$,
the reduction $\varphi_v$ of $\varphi$ at $v$
is not conjugate to $t^m$ or $D_m$.
\end{proposition}

\begin{proof}
After replacing $K$ by a finite extension, we may assume
that $K$ contains all $(m-1)$-st roots of unity.
All coefficients of $\varphi$ are $v$-adic integers
at all but finitely many places $v$.  For any such place,
write $k_v$ for the residue field and
$\varphi_v$ for the reduction of $\varphi$.
If $\varphi_v$ is conjugate
to the reduction $f_v$ of either $f=D_m$
or $f(t)=t^m$, write
$\varphi_v(t)=\mu^{-1}_v\circ f_v\circ\mu_v$
for some linear polynomial $\mu_v(t)=At+B\in k_v[t]$.
Because 
$\varphi_v$ and $f_v$ are both in normal form, we must have
$\mu_v(t)=\zeta_v t$, for some
$(m-1)$-st root of unity $\zeta_v\in k_v$.
(Indeed, because $\car K_v=\car E=0$ and both $\varphi_v$
and $f_v$ have trivial $t^{m-1}$ term, we must have $B=0$;
and because both are monic, $A$ must be an $(m-1)$-st root
of unity.)

Thus, at any such place $v$, $\varphi$
is congruent modulo $v$ to one of the $m$
polynomials
$\zeta^{-1}D_m(\zeta t)$ or
$\zeta^{-1}(\zeta t)^m=t^m$, where $\zeta\in K$
is an $(m-1)$-st root of unity.
Since $\varphi$ is not one of those $m$ polynomials
itself,
there are only finitely many such $v$ at which
that occurs.
\end{proof}

\begin{proposition}
\label{compositional powers}
Let $E$ be a field, and $K$ a function field of transcendence degree
$1$ over $E$. Let $f\in K[t]$ be an indecomposable polynomial of degree
greater than one. Then for all but finitely many places $v$ of $K$,
the reduction of $f$ modulo $v$ is also an indecomposable polynomial
over $\overline{k}_v$ of degree greater than one, where $k_v$ is the
residue field of $K$ at $v$.
\end{proposition}

\begin{proof}
First we note that for all but finitely many places $v$ of $K$,
the coefficients of $f$ are integral at $v$, and
the leading coefficient of $f$ is a unit at $v$.
Thus, the reduction
$f_v$ of $f$ modulo $v$ is a polynomial of same degree as $f$.

We will show that for any given positive integers $m$ and $n$ (with $m,n\ge 2$) such that $mn=\deg(f)$, if $f$ is not a composition of a polynomial of degree $m$ with a polynomial of degree $n$, then for all but finitely many places $v$ of $K$, the reduction of $f$ modulo $v$ cannot be written as a composition of two polynomials of degrees $m$ and $n$, respectively, with coefficients in $\overline{k}_v$. Because there are finitely many pairs of positive integers $(m,n)$ such that $mn=\deg(f)$, our desired conclusion follows.

Let $m$ and $n$ be positive integers such that $mn=\deg(f)$ (with $m,n\ge 2$). Then the nonexistence of polynomials 
$$g(t)=\sum_{i=0}^m a_i t^i\text{ and }h(t)=\sum_{j=0}^n b_jt^j$$ 
with coefficients in $\Kbar$,
such that $f=g\circ h$, where
$f(t)=\sum_{\ell=0}^{\deg(f)} c_{\ell} t^{\ell}$,
translates to the statement that the variety $X\subset
\bA^{m+n+2}$ given by the equations which must be satisfied by
the $a_i$'s and the $b_j$'s has
no $\overline{K}$-points. Furthermore, $X$ is a variety
defined over a subring $R$ of $K$ such that all but finitely many
places of $K$ are maximal ideals of $R$.

Suppose there is an infinite set $\cS$ of places $v$ of $K$
at which
$f_v$ is actually a composition of two polynomials of degrees $m$
and $n$, with coefficients in $\overline{k}_v$. Then the special fibre
of $X$ over $v$ is nonempty over $\overline{k}_v$ for each $v\in\cS$. Therefore, the
equations defining $X$ determine a nonempty locus over the
ultraproduct $\cK_{\cS,\cU}$ of all the infinitely many fields
$\overline{k}_v$ with respect to a non-principal ultrafilter $\cU$
based on $\cS$. However, $K$ embeds into $\cK_{\cS,\cU}$ (see
\cite[p. 198--199]{Udi}).  Since $X$ is defined over $K$ and has
a rational point over a field containing $K$, it must in fact
have an algebraic point over $K$, giving
a contradiction to the fact
that $X(\Kbar)$ is empty.
\end{proof}

We are ready to prove Theorem~\ref{polynomials, plane curves over C}.
\begin{proof}[Proof of Theorem~\ref{polynomials, plane curves over C}.]
If $\varphi$ is a linear polynomial, then the result follows from \cite{Bell}.
If $\varphi$
is conjugate to $t^m$ or $D_m$,
then our conclusion follows from Proposition~\ref{twisted Laurent}.
We may therefore
assume that $\varphi$ is an indecomposable, nonlinear polynomial which
is neither a conjugate of $t^m$, nor of $D_m$.
Furthermore, after conjugating $\varphi$ by a linear
polynomial $\mu$ (and replacing $(x_0,y_0)$ by
$(\mu^{-1}(x_0),\mu^{-1}(y_0))$ and
$C$ by $(\mu^{-1},\mu^{-1})(C)$),
we may assume that $\varphi$ is in normal form.
Let $m=\deg\varphi$.

As before, we may assume that $C$ is an
irreducible curve, and that $C$ does not project to a
single point to any of the coordinates.  For example,
if $C=\bA^1\times \{y_1\}$,
then $y_1$ is $\varphi$-periodic, and hence $C$
is $\Phi$-periodic. In particular,
we may assume that neither $x_0$ nor $y_0$ is
$\varphi$-preperiodic.

Let $K$ be a finitely generated field over which $C$, $\varphi$, $x_0$
and $y_0$ are defined. Furthermore, at the expense of replacing $K$ by
a finite extension, we may assume
that $C$ is geometrically irreducible
and that $K$ contains all critical points of $\varphi$
and all $(m-1)$-st roots of unity.

We will prove Theorem~\ref{polynomials, plane curves over C} by
induction on $d:=\trdeg_{\bQ}K$. If $d=0$, then $K$ is a number field,
and our conclusion follows from Theorem~\ref{plane curves}.

Assume $d\ge 1$. Then $K$ may be viewed as the function field of a
smooth, geometrically irreducible curve $Z$ defined over a finitely
generated field $E$; thus, $\trdeg_{\bQ}E=d-1$. Moreover, the curve $C$
extends to a $1$-dimensional scheme over $Z$ (called $\cC$), all
but finitely many of whose fibres $\cC_{\gamma}$ are irreducible curves.

We claim that
there are infinitely many places $\gamma$ of $K$ for which all of the
following statements hold.  (By a place of $K$, we mean a valuation of
the function field $K/E$, cf. Chapter~2 of \cite{serre-book}.)
\begin{itemize}
\item[(a)] The fibre $\cC_{\gamma}$ is an irreducible curve defined
over the residue field $E(\gamma)$ of $\gamma$, of the same degree as
$C$.
\item[(b)] All nonzero coefficients of $\varphi$ are units at the
place $\gamma$; in particular, $\varphi$ has good reduction at $\gamma$,
and so we write $\varphi_{\gamma}$ and 
$\Phi_{\gamma}:=(\varphi_{\gamma},\varphi_{\gamma})$ for
the reductions of $\varphi$ and $\Phi$ at $\gamma$.
\item[(c)] The critical points of
$\varphi_{\gamma}$ are reductions at $\gamma$ of the critical points
of $\varphi$.
\item[(d)] For each critical point $z$ of $\varphi$ (other than
infinity), the reduction $z_{\gamma}$ is not a periodic point for
$\varphi_{\gamma}$.
\item[(e)] The map $\cO\lra\cO_{\gamma}$ from the
$\Phi$-orbit of $(x_0,y_0)$ to the $\Phi_{\gamma}$-orbit of
$\left(x_{0,\gamma},y_{0,\gamma}\right)$,
induced by reduction at $\gamma$, is injective.
\item[(f)] $\varphi_{\gamma}$ is not conjugate to $t^m$ or $D_m$.
  (Recall $m=\deg\varphi$.)
\item[(g)] $\varphi_{\gamma}$ is a nonlinear, indecomposable polynomial.
\end{itemize}
Conditions (a)--(c) above are satisfied at all but finitely many
places $\gamma$ of $K$.  The same is true of conditions~(f)--(g),
by Proposition~\ref{needed conjugate reduction}
and Proposition~\ref{compositional powers}.
Condition~(d) for preperiodic (but not periodic) critical
points also holds at all but finitely many places;
see the first paragraph of the proof of Lemma~\ref{sil}.
Meanwhile, \cite[Proposition $6.2$]{Mike} says that the
reduction of any finite set of nonpreperiodic points remains
nonpreperiodic at infinitely many places $\gamma$ (in fact, at
all $\gamma$ on $Z$ of sufficiently large Weil height).
Thus, conditions~(d)--(e) hold by applying
\cite[Proposition $6.2$]{Mike} to $(x_0,y_0)$ and the
nonpreperiodic critical points, proving the claim.

Let $\gamma$ be one of the infinitely many places satisfying
conditions (a)--(g) above.
From condition (e), we deduce
that $\cC_{\gamma}(E(\gamma))\cap\cO_{\gamma}$ is infinite.
Conditions
(c)--(d) guarantee that $\varphi_{\gamma}$ has no periodic critical
points (other than the exceptional point at infinity).
Because $E(\gamma)$ is
a finite extension of $E$, we get $\trdeg_{\bQ}E(\gamma)=d-1$.
By the inductive hypothesis, then, $\cC_{\gamma}$
is $\Phi_{\gamma}$-periodic. By conditions (f)--(g)
and Theorem~\ref{conj: periodic curves}, $\cC_{\gamma}$ is the
zero set of an equation from one of the four forms (i)--(iv) in
Theorem~\ref{conj: periodic curves}.
In fact,
if $\varphi$ has type $(a,b)$, then
the degree $d$ in Theorem~\ref{conj: periodic curves}
satisfies $d|b$ and $\gcd(d,a)=1$,
because condition~(b) implies that $\varphi_{\gamma}$
also has type $(a,b)$.
Thus,
for one of the four forms (i)--(iv), there are infinitely
many places $\gamma$ satisfying (a)--(g) above
such that the equation for
$\cC_{\gamma}$ is of that form.
By symmetry, it suffices to consider only forms
(i) and (iii).

\emph{Case 1.} Assume there are infinitely many $\gamma$ satisfying
(a)--(g) such that $\cC_{\gamma}$ is given by an equation
$x=x(\gamma)$, for some $\varphi_{\gamma}$-periodic point
$x(\gamma)\in E(\gamma)$. Then, since the degree of $C$ is preserved by the reduction at $\gamma$, we see that the degree of $C$ must be $1$. Thus, $C$ is defined by an equation of the form $ax+by+c=0$. Since there are infinitely many $\gamma$ such that the above equation reduces at $\gamma$ to $x=x(\gamma)$, we must have $b=0$; hence, the curve $C$ must be given by an equation $x=x_1$ for some $x_1\in K$, contradicting our assumption that $C$ does not project to a point in any of the coordinates.

\emph{Case 2.} Assume there are infinitely many $\gamma$ satisfying
(a)--(g) such that $\cC_{\gamma}$ is given by an equation
$y=\zeta\varphi_{\gamma}^r(x)$, for some $r\ge 0$ and some
$d$-th root of unity $\zeta$, where $d\mid b$ and $\gcd(d,a)=1$.
Because there are only finitely many $b$-th roots of unity,
we may assume $\zeta$ is the same for all of the infinitely
many $\gamma$.  Moreover,
because $\cC_{\gamma}$ has the same degree as $C$, the integer $r$ is
the same for all such $\gamma$.
Thus,
there are infinitely many
places $\gamma$ for which the polynomial equation for $C$ reduces
modulo $\gamma$ to $y-\zeta\varphi^r(x)$,
and hence the two polynomials are the same.  Thus, $C$ is the
zero set of the polynomial $y-\zeta\varphi^r(x)$.
Because $\varphi$ is of type $(a,b)$, it follows that
$C$ is $\Phi$-periodic.
\end{proof}

Arguing precisely as in the proof of Theorem~\ref{curves},
Theorem~\ref{polynomials, arbitrary curves over C} follows as a consequence
of Theorem~\ref{polynomials, plane curves over C}.

\begin{remark}
\label{possible extensions}
In personal communications, Medvedev and Scanlon told us that,
using the methods of \cite{Medvedev/Scanlon},
it is possible to prove the conclsion of 
Theorem~\ref{conj: periodic curves}
even for decomposable polynomials $f$ that are not
compositional powers of other polynomials.
Using that stronger result in our proofs above, we could
then extend
Theorems~\ref{polynomials, plane curves over C}
and~\ref{polynomials, arbitrary curves over C}
to any $f$ that is not a compositional power of another
polynomial.  It would then be easy to extend those
results to
{\em all} polynomials $f\in\bC[t]$ (with no periodic superattracting
points other than exceptional points);
indeed, if $f=g^k$ is a compositional power, then we may simply
replace the action of $f$ with the action of $g$.
\end{remark}

\def\cprime{$'$} \def\cprime{$'$} \def\cprime{$'$} \def\cprime{$'$}
\providecommand{\bysame}{\leavevmode\hbox to3em{\hrulefill}\thinspace}
\providecommand{\MR}{\relax\ifhmode\unskip\space\fi MR }
\providecommand{\MRhref}[2]{%
  \href{http://www.ams.org/mathscinet-getitem?mr=#1}{#2}
}
\providecommand{\href}[2]{#2}



\begin{thebibliography}{Ray83b}

\bibitem[Bea90]{Beardon-2}
A.~F.~Beardon, \emph{Symmetries of Julia sets}, Bull. London Math. Soc. \textbf{22} (1990), no.~6, 576--582.

\bibitem[Bea91]{beardon-book}
A.~F.~Beardon, \emph{Iteration of rational functions}, Graduate Texts in
  Mathematics, vol. 132, Springer-Verlag, New York, 1991, Complex analytic
  dynamical systems. 

\bibitem[Bel06]{Bell}
J.~P. Bell, \emph{A generalised {S}kolem-{M}ahler-{L}ech theorem for affine
  varieties}, J. London Math. Soc. (2) \textbf{73} (2006), no.~2, 367--379.

\bibitem[BGT08]{BGT}
J.~P. Bell, D~Ghioca, and T.~J. Tucker.
\newblock The dynamical {M}ordell-{L}ang problem for \'etale maps.
\newblock Submitted for publication, available at {\tt
  arxiv.org/abs/0808.3266}, 19 pages, 2008.


\bibitem[Den94]{Denis-dynamical}
L.~Denis, \emph{G\'{e}om\'{e}trie et suites r\'{e}currentes}, Bull. Soc. Math.
  France \textbf{122} (1994), no.~1, 13--27.

\bibitem[Ere90]{Eremenko}
A.~\`{E}. Eremenko, \emph{On some functional equations connected with the
  iteration of rational functions}, Leningrad Math. J. \textbf{1} (1990),
  no.~4, 905--919.

\bibitem[ESS02]{ESS}
J.-H. Evertse, H.~P. Schlickewei, and W.~M. Schmidt, \emph{Linear equations in
  variables which lie in a multiplicative group}, Ann. of Math. (2)
  \textbf{155} (2002), no.~3, 807--836.

\bibitem[Fal94]{Faltings}
G.~Faltings, \emph{The general case of {S}. {L}ang's conjecture}, Barsotti
  Symposium in Algebraic Geometry (Abano Terme, 1991), Perspect. Math., no.~15,
  Academic Press, San Diego, CA, 1994, pp.~175--182.

\bibitem[Fat21]{Fatou-1}
P.~Fatou, \emph{Sur les fonctions qui admettent plusieurs th\`{e}or\`{e}mes de
  multiplication}, C. R. Acad. Sci. Paris S\'{e}r. I Math. \textbf{173} (1921),
  571--573.

\bibitem[Fat23]{Fatou-2}
\bysame, \emph{Sur l'iteration analytique et les substitutions permutables}, J.
  Math. \textbf{2} (1923), 343.

\bibitem[GT]{GT-newlog}
D.~Ghioca and T.~J. Tucker, \emph{Periodic points, linearizing maps, and the dynamical Mordell-Lang problem}, to appear in J. Number Theory, 15 pages.

\bibitem[GTZ08]{Mike}
D.~Ghioca, T.~J. Tucker, and M.~E. Zieve, \emph{Intersections of polynomial
  orbits, and a dynamical {M}ordell-{L}ang conjecture}, Invent.
  Math. \textbf{171} (2008), 463--483.

\bibitem[Hru98]{Udi}
E.~Hrushovski, \emph{Proof of Manin's theorem by reduction to positive characteristic}, Model theory and algebraic geometry, 197--205, Lecture Notes in Math. \textbf{1696}, Springer, Berlin, 1998.

\bibitem[Jon08]{Jones} R.~Jones, \emph{The density of prime divisors
    in the arithmetic dynamics of quadratic polynomials},
  J. Lond. Math. Soc. (2) \textbf{78} (2008), no.~2, 523--544.


\bibitem[Jul22]{Julia}
G.~Julia, \emph{M\'{e}moire sur la permutabilit\'{e} des fractions
  rationnelles}, Ann. Sci. \'{E}cole Norm. Sup. \textbf{39} (1922), 131--215.

\bibitem[Lec53]{Lech}
C.~Lech, \emph{A note on recurring series}, Ark. Mat. \textbf{2} (1953),
  417--421.

\bibitem[Mah35]{Mahler-2}
K.~Mahler, \emph{Eine arithmetische {E}igenshaft der {T}aylor-{K}oeffizienten
  rationaler {F}unktionen}, Proc. Kon. Nederlandsche Akad. v. Wetenschappen
  \textbf{38} (1935), 50--60.



\bibitem[MS]{Medvedev/Scanlon}
A.~Medvedev and T.~Scanlon, \emph{Polynomial dynamics}, preprint, available on {\tt http://arxiv.org/abs/0901.2352}, 67 pages.

\bibitem[MS94]{MorSil1}
P.~Morton and J.~H. Silverman, \emph{Rational periodic points of rational
  functions}, Internat. Math. Res. Notices (1994), no.~2, 97--110.

\bibitem[MS95]{MorSil2}
\bysame, \emph{Periodic points, multiplicities, and dynamical units}, J. Reine
  Angew. Math. \textbf{461} (1995), 81--122.

\bibitem[Ray83a]{Raynaud1}
M.~Raynaud, \emph{Courbes sur une vari\'{e}t\'{e} ab\'{e}lienne et points de
  torsion}, Invent. Math. \textbf{71} (1983), no.~1, 207--233.

\bibitem[Ray83b]{Raynaud2}
\bysame, \emph{Sous-vari\'{e}t\'{e}s d'une vari\'{e}t\'{e} ab\'{e}lienne et
  points de torsion}, Arithmetic and geometry, vol. I, Progr. Math., vol.~35,
  Birkh\"{a}user, Boston, MA, 1983, pp.~327--352.

\bibitem[RL03]{Riv}
J.~Rivera-Letelier, \emph{Dynamique des fonctions rationnelles sur des corps
  locaux}, Ast\'erisque (2003), no.~287, 147--230, Geometric methods in
  dynamics. II.

\bibitem[Rob00]{Robert}
A.~M. Robert, \emph{A course in {$p$}-adic analysis}, Graduate Texts in
  Mathematics, vol. 198, Springer-Verlag, New York, 2000.

\bibitem[Ser97]{serre-book}
Jean-Pierre Serre, \emph{Lectures on the {M}ordell-{W}eil theorem}, third ed.,
  Aspects of Mathematics, Friedr. Vieweg \& Sohn, Braunschweig, 1997,
  Translated from the French and edited by Martin Brown from notes by Michel
  Waldschmidt, With a foreword by Brown and Serre. 

\bibitem[Sha77]{Shaf}
I.~R. Shafarevich, \emph{Basic algebraic geometry}, study ed., Springer-Verlag,
  Berlin, 1977, Translated from the Russian by K. A. Hirsch, Revised printing
  of Grundlehren der mathematischen Wissenschaften, Vol. 213, 1974.

\bibitem[Sil93]{SilSiegel}
J.~H. Silverman, \emph{Integer points, {D}iophantine approximation, and
  iteration of rational maps}, Duke Math. J. \textbf{71} (1993), no.~3,
  793--829.

\bibitem[Sko34]{Skolem}
T.~Skolem, \emph{Ein {V}erfahren zur {B}ehandlung gewisser exponentialer
  {G}leichungen und diophantischer {G}leichungen}, C. r. 8 congr. scand. \`{a}
  Stockholm (1934), 163--188.

\bibitem[Ull98]{Ullmo}
E.~Ullmo, \emph{Positivit\'e et discr\'etion des points alg\'ebriques des
  courbes}, Ann. of Math. (2) \textbf{147} (1998), no.~1, 167--179.

\bibitem[Voj96]{V1}
P.~Vojta, \emph{Integral points on subvarieties of semiabelian varieties. {I}},
  Invent. Math. \textbf{126} (1996), no.~1, 133--181.

\bibitem[Zha98]{Zhang}
S.~Zhang, \emph{Equidistribution of small points on abelian varieties}, Ann. of
  Math. (2) \textbf{147} (1998), no.~1, 159--165.

\bibitem[Zha06]{ZhangLec}
S.~Zhang, \emph{Distributions in {A}lgebraic {D}ynamics}, Survey in
  Differential Geometry, vol.~10, International Press, 2006, pp.~381--430.

\end{thebibliography}

\end{document}